\documentclass[reqno, 12pt]{amsart}

\usepackage[algoruled,lined,noresetcount,norelsize]{algorithm2e}

\usepackage{graphicx}
\usepackage{amsmath,amssymb,amsthm}
\usepackage{mathrsfs}
\usepackage{mathabx}\changenotsign
\usepackage{dsfont}
\usepackage{xcolor}
\usepackage[backref]{hyperref}
\hypersetup{
    colorlinks,
    linkcolor={red!60!black},
    citecolor={green!60!black},
    urlcolor={blue!60!black}
}
\usepackage[T1]{fontenc}
\usepackage{lmodern}
\usepackage[babel]{microtype}
\usepackage[english]{babel}

\usepackage{geometry}
\numberwithin{equation}{section}
\numberwithin{figure}{section}

\usepackage{enumitem}

\theoremstyle{plain}
\newtheorem{thm}{Theorem}[section]
\newtheorem{prop}[thm]{Proposition}

\newtheorem{cor}[thm]{Corollary}
\newtheorem{lemma}[thm]{Lemma}

\theoremstyle{definition}
\newtheorem{rem}[thm]{Remark}
\newtheorem{dfn}[thm]{Definition}

\renewcommand{\log}{\ln}

\newcommand{\gnp}{G_{n,p}}

\newcommand\eps{\varepsilon}
\newcommand{\Exp}{\mathbb{E}}
\newcommand{\Prob}{\mathbb{P}}

\newcommand{\E}{\mathcal{E}}

\newcommand{\maker}{\mathcal{M}}
\newcommand{\breaker}{\mathcal{B}}
\newcommand{\waiter}{\mathcal{W}}
\newcommand{\client}{\mathcal{C}}

\newcommand\dang{\operatorname{dang}}

\begin{document}

\title[Positional games on randomly perturbed graphs]{Positional games on randomly perturbed graphs}

\author[Dennis Clemens]{Dennis Clemens}
\author[Fabian Hamann]{Fabian Hamann}
\author[Yannick Mogge]{Yannick Mogge}

\address{(DC, FH, YM) Hamburg University of Technology, Institute of Mathematics, Am Schwarzenberg-Campus 3, 21073 Hamburg, Germany }
\email{dennis.clemens@tuhh.de, fabian.hamann@tuhh.de, yannick.mogge@tuhh.de}

\author[Olaf Parczyk]{Olaf Parczyk}

\address{(OP) London School of Economics, Department of Mathematics, London, WC2A 2AE, UK.}
\email{o.parczyk@lse.ac.uk}

\thanks{OP was supported by the DFG (Grant PA 3513/1-1)}


\begin{abstract}
Maker-Breaker games are played on a hypergraph $(X,\mathcal{F})$, where $\mathcal{F} \subseteq 2^X$ denotes the family of winning sets. Both players alternately claim a predefined amount of edges (called bias) from the board $X$, and Maker wins the game if she is able to occupy any winning set $F \in \mathcal{F}$. These games are well studied when played on the complete graph $K_n$ or on a random graph $G_{n,p}$. In this paper we consider Maker-Breaker games played on randomly perturbed graphs instead. These graphs consist of the union of a deterministic graph $G_\alpha$ with minimum degree at least $\alpha n$ and a binomial random graph $G_{n,p}$. 
Depending on $\alpha$ and Breaker's bias $b$ we determine the order of the threshold probability for winning the Hamiltonicity game and the $k$-connectivity game on $G_{\alpha}\cup G_{n,p}$, and we discuss the $H$-game when $b=1$. Furthermore, we give optimal results for the Waiter-Client versions of all mentioned games.
\end{abstract}

\maketitle


\section{Introduction}

In general, a positional game is a perfect information game played by two players on a hypergraph $(X,\mathcal{F})$. 
Throughout the game both players occupy elements of the {\em board} $X$ according to some predefined rule,
and the winner is determined through the
{\em family of winning sets} $\mathcal{F}$.
Research of the last decades has generated many interesting results in the area of positional games (see e.g.~\cite{BeckBook,HKSS_PosGames}),
where plenty of different types were considered,
including {\em Maker-Breaker games} (see e.g.~ \cite{BL2000,FKN2015,GS_MinDegree,Kriv_Ham,liebenau2020threshold}),
{\em Waiter-Client games} (see e.g.~ \cite{BHKL2016,DK2016,HKT2017}),
and many more.
 
In this paper we are mainly interested in 
positional games where $X$ is the edge set of some given
graph $G$, and where $\mathcal{F}$ is the family of all Hamilton cycles of $G$,
all $k$-vertex-connected spanning subgraphs of $G$,
or all copies of some fixed graph $H$ in $G$, respectively. 
Before we will discuss such games 
further, let us first recall what is known about the appearance of these structures
in dense graphs, random graphs and randomly perturbed graphs. 

\subsection{Dense graphs} For a given graph $G$
let $\delta(G)$ denote its minimum degree.
When we consider all graphs $G$ on $n$ vertices, it is quite natural to ask how large $\delta(G)$ needs to be
in order to guarantee the existence of a given
structure. Dirac~\cite{D_HAM} proved that any graph on $n \ge 3$ vertices with minimum degree at least $\frac{n}{2}$ is Hamiltonian, with the bound being sharp, since the vertex disjoint union of two cliques of sizes $\lfloor \frac{n}{2} \rfloor$ and $\lceil \frac{n}{2} \rceil$
does not contain a Hamilton cycle.
Moreover, it is an easy observation, that for any positive integer $k$, any $n$-vertex graph with minimum degree at least $\frac{n+k-1}{2}$ is $k$-vertex-connected.
Again this bound is seen to be sharp
by looking at the following example:
take the vertex disjoint union of
three cliques of sizes $k-1$, $\lfloor \frac{n-k+1}{2} \rfloor$ and $\lceil \frac{n-k+1}{2} \rceil$, respectively,
and add all edges between the first clique
and both of the other cliques.
For the containment of a fixed graph $H$ a sufficient minimum degree condition can be derived from Tur\'an's Theorem~\cite{Turan} with regularity, or from the following theorem of Erd\H{o}s and Stone~\cite{ES_extremal}:
	for any fixed graph $H$ with chromatic number $r>2$, any $n$-vertex graph with minimum degree $\left(\frac{r-2}{r-1} +o(1)\right) n$ contains a copy of $H$ provided that $n$ is large.
	This bound is seen to be sharp by considering
a Tur\'an graph on $n$ vertices with $r-1$ vertex classes, i.e.~an $(r-1)$-partite graph with 
all class sizes differing at most by 1.
	For bipartite graphs $H$ any linear minimum degree is sufficient.

\subsection{Random graphs}

Let $G_{n,p}$ be the binomial random graph model, i.e.~the model of $n$-vertex random graphs, where each edge is present with probability $p$ independently of all others.
For simplicity, we will often write $\gnp$ also for a graph $G \sim \gnp$ drawn according to this distribution.
Given some graph property $\mathcal{P}$,
we are interested in a probability function $\hat{p} = \hat{p}(\mathcal{P})$ such that
the following holds: when $p = \omega(\hat{p})$, then $G_{n,p}$ satisfies $\mathcal{P}$ asymptotically almost surely (a.a.s.), i.e.~with probability tending to one as $n$ tends to infinity.
When additionally it holds that $G_{n,p}$ a.a.s.~does not have property $\mathcal{P}$
when $p=o(\hat{p})$, we call $\hat{p}$ a \emph{threshold probability} for the property
$\mathcal{P}$. The existence of such a threshold is known for all monotone graph properties $\mathcal{P}$~\cite{Bollobas}.
If the same conclusion even holds with $p \ge (1+\eps) \hat{p}$ and $p \le (1-\eps) \hat{p}$ for any constant $\eps > 0$, then we call $\hat{p}$ a \emph{sharp threshold probability} for the property $\mathcal{P}$.

Pos\'a~\cite{P_HAM} and Korshunov~\cite{K_HAM} independently proved that $G_{n,p}$ has a sharp threshold for containing a Hamilton cycle at $\frac{\log n}{n}$.
In fact, there are even more precise results known, e.g.~it holds a.a.s. that at the point when $G_{n,p}$ has minimum degree two it is already Hamiltonian~\cite{B_HAM}.
Note that the existence of a Hamilton cycle immediately implies $2$-vertex-connectivity. More generally, it follows from an early work of Erd\H{o}s and R\'enyi~\cite{ER_connected}, that for any fixed positive integer $k$ the random graph $G_{n,p}$ has a sharp threshold for being
$k$-vertex-connected at $\frac{\log n}{n}$.
Lastly, consider the containment of a fixed graph $H$.
Let $d(H)=\frac{e(H)}{v(H)}$ and denote by
	\[m(H)= \max_{H' \subseteq H, v(H)>0} d(H)\]
	the \emph{maximum subgraph density of $H$}. 
Bollob\'as~\cite{Bollobas} proved that 
the threshold probability for
$G_{n,p}$ containing a copy of $H$ lies at $n^{-1/m(H)}$.

\subsection{Randomly perturbed graphs}
Note that all of the constructions described above for showing the sharpness of the minimum degree conditions
for dense graphs are 'well-structured' in a certain way: they contain very dense parts (e.g.~large cliques) and very sparse parts (e.g.~large independent sets)
at the same time,
and thus they are far away from the 'typical shape' of a random graph. For that reason, one may hope for improving the minimum degree conditions, that we already discussed for dense graphs,
when we allow to extend the given graph slightly by 
additionally sprinkling a few random edges on top of it.
To be more precise, fix any real $\alpha>0$ and let 
$G_\alpha$ be a sequence of $n$-vertex graphs with minimum degree at least $\alpha n$. Our goal is to consider
the model $G_\alpha \cup G_{n,p}$ which was first studied by Bohman, Frieze, and Martin~\cite{BFM_HAM}.
In the following we slightly abuse notation and also write $G_\alpha \cup G_{n,p}$ for $G_\alpha \cup G$ with $G \sim \gnp$.
Similarly to the discussion of the binomial random graph, we can look for $\hat{p}=\hat{p}(\mathcal{P})$
such that the following holds:
\begin{itemize}
\item with $p = \omega(\hat{p})$ it is true that for any sequence $G_\alpha$ the desired property $\mathcal{P}$ holds a.a.s.~in $G_\alpha \cup G_{n,p}$,
\item with $p = o(\hat{p})$ there exists a sequence
$G_\alpha$ such that the desired property $\mathcal{P}$ a.a.s.~does not hold in $G_\alpha \cup G_{n,p}$.
\end{itemize}

Bohman, Frieze, and Martin~\cite{BFM_HAM}
showed that for any $\alpha>0$ there exists a large enough constant $C$ such that it is sufficient to take $p \ge \frac{C}{n}$ to ensure that a.a.s.~$G_\alpha \cup G_{n,p}$ is Hamiltonian.	 
	This is asymptotically optimal, because with $G_\alpha=K_{\alpha n,(1-\alpha n)}$ we can use at most $2 \alpha n$ edges of $G_\alpha$ for a Hamilton cycle and when $p=o\left(\frac{1}{n}\right)$ we a.a.s.~have that $G_{n,p}$
	only adds $o(n)$ edges.
On one hand, this result shows that any minimum degree linear in $n$ suffices to ensure that a.a.s. a Hamilton cycle appears when a sparse random graph with $\omega(n)$ edges is added. On the other hand, we see that starting with a dense graph $G_{\alpha}$ we need an edge probability $p$ for the likely Hamiltonicity of $G_{\alpha}\cup G_{n,p}$ which is a $\ln$-factor smaller than the corresponding threshold for the binomial random graph. 

Among other things, it was shown by Bohman, Frieze, Krivelevich, and Martin~\cite{BFKM2004} that for any fixed positive integer $k$ the randomly perturbed graph $G_\alpha \cup G_{n,p}$ is $k$-vertex-connected when $p = \omega\left(\frac{1}{n^2}\right)$; results are also given for the case when $k$ is not a constant but depending on $n$.

To discuss the results for a fixed graph $H$ we first need to define the \emph{$r$-partite density} of $H$, as it was introduced by Krivelevich, Sudakov, and Tetali~\cite{KST_SmoothedAnalysis}.
	Let
	\[m^{(r)}(H):= \min_{V(H) = \cup_iP_i} \max_i m(H[P_i]),\]
	where we minimise over all possible partitions of the vertex set of $H$ into $r$ parts $P_1,\dots,P_r$.	
	Then, for any integer $r$, any $\alpha \in ( \frac{r-2}{r-1},\frac{r-1}{r} ]$, and any fixed graph $H$ with $m^{(r)}(H)>0$, the randomly perturbed graph $G_\alpha \cup G_{n,p}$ a.a.s.~contains a copy of $H$ provided that $p = \omega\left(n^{-1/m^{(r)}(H)}\right)$~\cite{KST_SmoothedAnalysis}.
	Note that when $m^{(r)}(H)=0$, this implies, that $\chi(H)\leq r$, and thus, due to Erd\H{o}s and Stone~\cite{ES_extremal}, there already exists a copy of $H$ in $G_\alpha$ alone.
	To see that this is optimal it suffices to consider a complete $r$-partite graph on $n$ vertices with roughly equal parts and to note that in one of the parts there has to be a copy of some $H' \subseteq H$ with $m(H') \ge m^{(r)}(H)$ fully composed of edges of $G_{n,p}$, which a.a.s.~does not appear when $p = o\left(n^{-1/m^{(r)}(H)}\right)$.

\subsection{Maker-Breaker games} In our paper we will mainly focus on Maker-Breaker games.
A \linebreak $(1:b)$ Maker-Breaker game (also referred to as $b$-{\em biased Maker-Breaker game}) on some hypergraph $(X,\mathcal{F})$ is played as follows.
Maker and Breaker alternate in taking $1$ and $b$ unclaimed elements of $X$, respectively (except for maybe the last round where Breaker could take less elements), and the game is Maker's win if she fully claims an element of $\mathcal{F}$; otherwise Breaker wins. Since increasing the bias $b$ is never a disadvantage for Breaker, it can easily be shown that for any hypergraph $(X,\mathcal{F})$ there must be a 
{\em threshold bias} $b_{\mathcal{F}}$ such that
Maker wins if and only if $b<b_{\mathcal{F}}$ (see e.g.~\cite{HKSS_PosGames}).

In our setting $X$ will be the edge set of some
host graph $G$ and $\mathcal{F}$ will consist of all Hamilton cycles, $k$-vertex-connected subgraphs, or copies of $H$ in $G$, respectively.
In the literature, $G$ is usually chosen to be the complete graph $K_n$, or (more recently) the binomial random graph $G_{n,p}$ (see e.g.~\cite{HKSS2009}, \cite{SS2005}).
But already when playing on $K_n$ the intuition from random graphs plays an important role, since Maker's subgraph can exhibit similar properties as a subgraph where each edge of $X=E(K_n)$ is taken at random with probability $\frac{1}{b+1}$.
Moreover, for many (but not all) families $\mathcal{F}$ 
the threshold bias in the $(1:b)$ game on $X=E(K_n)$
is tightly linked to the threshold probability for $G_{n,p}$ to be such that Maker wins the $(1:1)$ game
with winning sets $\mathcal{F}$.

As a first example, let us look at the Maker-Breaker 
{\em Hamiltonicity game},
where the family of winning sets $\mathcal{F}=\mathcal{H}(G)$ consists of all Hamilton cycles
of the host graph $G$ which the game is played on.
When playing on $G=K_n$, it was shown by Krivelevich~\cite{Kriv_Ham} that the threshold bias
$b_{\mathcal{H}(K_n)}$
is of size $(1+o(1))\frac{n}{\ln n}$. 
One interesting fact about this result is that it supports the {\em probabilistic intuition}:
if Maker and Breaker would play the $b$-biased game fully at random, Maker's graph would behave similarly
to $G_{n,p}$ with edge probability $p=\frac{1}{b+1}$, which asymptotically equals the threshold $(1+o(1))\frac{\ln n}{n}$
for containing a Hamilton cycle when $b=b_{\mathcal{H}(K_n)}$.
Furthermore, as shown by Krivelevich, Lee, and Sudakov~\cite{KLS_dirac}, Maker can even win the $b$-biased Hamiltonicity game on a graph $G$ with $\delta(G)\geq \frac{n}{2}$ provided that $b \le \frac{cn}{\ln n}$ for some small enough constant $c>0$.
On the other hand~\cite{HKSS2009}, $G_{n,p}$ is a.a.s.~such that Maker wins the $(1:1)$ Maker-Breaker Hamiltonicity game provided that $p \ge (1+\eps) \frac{\ln n}{n}$ for any fixed $\eps>0$. This coincides with the appearance of a Hamilton cycle in $G_{n,p}$ as discussed above, and this threshold probability
happens to be asymptotically the inverse 
of the threshold bias $b_{\mathcal{H}(K_n)}$.
Moreover, for $p=\omega\left( \frac{\ln n}{n} \right)$
it has been shown that the threshold bias for winning
the Hamiltonicity game on $G_{n,p}$ is a.a.s.~of size
$\Theta\left(\frac{np}{\ln n}\right)$~\cite{FGKN_BiasedGames}, giving a linear dependency between the edge probability $p$
and the threshold bias $b_{\mathcal{H}(G_{n,p})}$.

Next let us consider the {\em $k$-vertex-connectivity game} where $\mathcal{F}=\mathcal{C}_k(G)$
consists of all spanning $k$-vertex-connected subgraphs of the host graph $G$ which the game is played on.
As observed by Krivelevich~\cite{Kriv_Ham},
his approach for the Hamiltonicity game on $K_n$
can be modified to show that 
$b_{\mathcal{C}_k(K_n)}=(1+o(1))\frac{n}{\ln n}$, thus again providing an example supporting the probabilistic intuition. Additionally and similarly as above, Maker wins the $(1:1)$ Maker-Breaker $k$-vertex-connectivity game on $G_{n,p}$ provided that $p \ge (1+\eps) \frac{\ln n}{n}$ for any fixed $\eps>0$~\cite{BFHK_MBhitting}; and
for $p=\omega\left( \frac{\ln n}{n} \right)$, the threshold bias for winning
the $k$-vertex-connectivity game on $G_{n,p}$ 
is a.a.s.~of size $\Theta\left(\frac{np}{\ln n}\right)$~\cite{FGKN_BiasedGames}.

Finally, for any fixed graph $H$, let us turn to the Maker-Breaker {\em $H$-game}, where
$\mathcal{F}=\mathcal{F}_H(G)$
consists of all copies of $H$ that are contained in the host graph $G$ which the game is played on.
For any graph $H'$ on at least 3 vertices, let $d_2(H')=\frac{e(H')-1}{v(H')-2}$, and denote by
\[m_2(H)=\max_{H' \subseteq H, v(H) >2} d_2(H')\]
the \emph{maximum $2$-density} of $H$, where (for convenience) we let $m_2(P_1)=d_2(P_1)=1$ and $m_2(P_0)=d_2(P_0)=0$ with $P_1$ being a single edge and $P_0$ being an isolated vertex.
For the $H$-game on $K_n$,
Bednarska and \L uczak~\cite{BL2000} proved that
$b_{\mathcal{F}_H(K_n)} = \Theta\left( n^{1/m_2(H)} \right)$, provided that $H$ contains at least $3$ non-isolated vertices.
Our proofs for the results discussed in the next section will also imply the following for the $(1:1)$ Maker-Breaker $H$-game on graphs with large minimum degree.
\begin{thm}
	\label{thm:hgame0}
	Let $r \ge 2$ be an integer, $\alpha \in ( \frac{r-2}{r-1},\frac{r-1}{r} ]$, let $H$ be a fixed graph with chromatic number $r$, and let $G_\alpha$ be a graph on $n$ vertices with $\delta(G_{\alpha}) \geq \alpha n$.
	Then provided that $n$ is large enough the following holds:
	playing a $(1:1)$ Maker-Breaker game on the edges 	
	of $G_{\alpha}$, Maker has a strategy to obtain a copy of $H$.
\end{thm}
	
Now, let $H$ be a graph for which there exists $H' \subseteq H$ such that $d_2(H')=m_2(H)$, $d_2(H'')<d_2(H')$ for all $H''\subsetneq H'$, and $H'$ is not a tree or triangle. Then
the threshold probability $p_{\mathcal{F}_H}$ for winning the Maker-Breaker $H$-game on $G_{n,p}$ turns out to be of the order
$n^{-1/m_2(H)}$~\cite{NSS2016}, and hence, similar to the discussion of the Hamiltonicity game and the $k$-vertex-connectivity game, it is asymptotically the inverse of the order of
the threshold bias $b_{\mathcal{F}_H}(K_n)$
for the corresponding game on $K_n$.
However, this last observation is not true for general $H$. For instance, when $H=K_3$, the threshold probability $p_{\mathcal{F}_{K_3}}$ is of order $n^{-5/9}$~\cite{SS2005}. Until today, it is still an open problem to find the threshold probability $p_{\mathcal{F}_H}$ for every $H$.

\subsection{Our results - Maker-Breaker games on randomly perturbed graphs}

The main goal of this paper is to study Maker-Breaker games on randomly perturbed graphs where the family of winning sets $\mathcal{F}$ consists of Hamilton cycles, $k$-vertex-connected spanning subgraphs or copies of a fixed graph $H$, respectively.
Depending on $\alpha$ and Breaker's bias $b$, we aim to determine the order of the threshold probability for winning
the $(1:b)$ Maker-Breaker game on $G_\alpha \cup G_{n,p}$ 
with winning sets $\mathcal{F}$. For the case when $b=1$, one main question will be, whether the appearance of the structures from $\mathcal{F}$ in $G_\alpha \cup G_{n,p}$ is roughly sufficient for a Maker's win in the corresponding game.

\medskip 

For the Hamiltonicity game we prove the following.

\begin{thm}[Biased MB Hamiltonicity game]
	\label{thm:Ham}
	For every real $\alpha>0$ there exist constants $c,C>0$ such that the following holds for large enough integers $n$. Let $G_{\alpha}$ be a graph on $n$ vertices with 
	$\delta(G_{\alpha})\geq \alpha n$, let 
	$b\leq \frac{cn}{\ln n}$ be an integer, and let $p \geq \frac{Cb}{n}$. Then a.a.s. the following holds: playing a $(1:b)$ Maker-Breaker game on the edges of $G_{\alpha}\cup G_{n,p}$, Maker has a strategy to claim a Hamilton cycle.
\end{thm}

Note that the bound on $b$ is optimal 
up to the constant factor,
since for $b\geq (1+\eps)\frac{n}{\ln n}$, Breaker can
isolate a vertex on any graph having $n$ vertices~\cite{CE1978}, and hence he wins the Hamiltonicity game.
Also, for $\alpha\in (0,\frac{1}{2})$, the bound on $p$ is optimal up to the constant factor.
To see this, it is enough to show that for each $\alpha \in (0,\frac{1}{2})$ there exists a sequence of graphs $G_\alpha$ such that Breaker a.a.s.~wins the Hamiltonicity game on $G_{\alpha}\cup G_{n,p}$ for $p \leq \frac{(1 - 2\alpha)b}{2(1-\alpha)^2n}$. For this, consider $G_\alpha$ to be a complete bipartite graph $A \cup B$ with $|A| = \alpha n$ and $|B| = (1 - \alpha)n$. Now, every Hamilton cycle in $G_{\alpha} \cup G_{n,p}$ needs to contain at least $(1-2\alpha)n$ edges within $B$. However, a.a.s.~$G_{n,p}$ has less than $\frac{(1-2\alpha)bn}{2}$ edges within $B$, and hence, Breaker can ensure that Maker cannot claim $(1-2\alpha)n$ of these edges by occupying $b$ of these edges for himself in each round. Note that in this case we obtain a linear dependency between Breaker's bias $b$
and the threshold probability $p$ for winning the game,
analogously to the Hamiltonicity game on $G_{n,p}$.
In the remaining case, when $\alpha \geq \frac{1}{2}$, we can actually choose $p=0$ and $b\leq \frac{cn}{\ln n}$ 
by the result of Krivelevich, Lee, and Sudakov~\cite{KLS_dirac}.

Moreover, note that the above theorem strengthens the result of Bohman, Frieze, and Martin~\cite{BFM_HAM}
on the containment of Hamilton cycles.
When $p\geq \frac{C}{n}$ for some large enough constant $C$, the graph $G_{\alpha}\cup G_{n,p}$ a.a.s.~does not
only contain a Hamilton cycle;
instead it is so rich of this structure that
Maker can win the $(1:1)$ Maker-Breaker Hamiltonicity game on it. 

\medskip

For the $k$-vertex-connectivity game we show the following result.

\begin{thm}[Biased MB $k$-vertex-connectivity game]
\label{thm:vertex-connectivity}
For every real $\alpha>0$ and every integer $k\geq 1$ there exist constants $C,c>0$ such that the following holds for large enough integers $n$. Let $G_{\alpha}$ be a graph on $n$ vertices with $\delta(G_{\alpha})\geq \alpha n$, let $b\leq \frac{cn}{\ln n}$ be an integer, and let $p\geq \frac{Cb}{n^2}$. Then with probability at least $1-\exp(-c pn^2)$ the following holds:  playing a $(1:b)$ Maker-Breaker game on the edges of $G_{\alpha}\cup G_{n,p}$, Maker has a strategy to claim a spanning $k$-vertex-connected graph.
\end{thm}

Note again that this theorem
strengthens the result on
the $k$-vertex-connectivity of $G_{\alpha}\cup G_{n,p}$
given in~\cite{BFKM2004}, and it also gives a linear dependency between $b$ and the threshold probability $p$.
The bound on $b$ is
optimal up to the constant factor by the same reason as
before.
For optimality regarding $p$, consider $G_{\alpha}$ to be a graph consisting of (roughly) 
$\frac{1}{\alpha}$ vertex disjoint cliques
of size (roughly) $\alpha n$. 
When $p\leq \frac{\eps b}{n^2}$, using Markov's inequality we see that with probability at least $1-\eps$ there are less than $b$ edges in the graph $G_{n,p}$ and therefore Breaker can easily ensure that Maker receives at most 1 such edge. However, adding this edge to $G_{\alpha}$ does not even result in a connected spanning graph and hence Maker cannot occupy such a graph 
in the game on $G_{\alpha}\cup G_{n,p}$.

\medskip

Finally, let us turn to the $H$-game on $G_{\alpha}\cup G_{n,p}$. For the case when $H$ is a clique we obtain the following result.

\begin{thm}[Unbiased MB Clique game]
\label{thm:clique_Maker}
	Let $\gamma>0$, integers $t,r \ge 2$, and $\alpha \in ( \frac{r-2}{r-1},\frac{r-1}{r} ]$ be given.
	Let $G_\alpha$ be an $n$-vertex graph with  
	$\delta(G_{\alpha})\geq \alpha n$, and let
	$p \ge n^{-2/(t+1)+\gamma}$. Then
	a.a.s.~the following holds:
	playing a $(1:1)$ Maker-Breaker game  on the edges 	
	of $G_{\alpha}\cup G_{n,p}$, Maker 
	has a strategy to obtain a copy of $K_{tr}$.
\end{thm}

For $t \ge 4$ this is optimal up to the constant $\gamma$ in the exponent, because with $p = o \left(n^{-2/(t+1)} \right)$ and $G_\alpha$ being an $r$-partite Tur\'{a}n graph we need a copy of $K_t$ on at least one of the partite sets, but a.a.s.~Breaker has a strategy to prevent Maker from having any copy of $K_t$ in $G_{n,p}$~\cite{NSS2016}.

To extend this to arbitrary graphs $H$ we introduce the following {\em $r$-partite $2$-density}, analogously to the $r$-partite density introduced by Krivelevich, Sudakov, and Tetali~\cite{KST_SmoothedAnalysis} for studying the appearance of $H$ in $G_\alpha \cup G_{n,p}$.
Let
\begin{align*}
	m_2^{(r)}(H):= \min_{V(H) = \cup_iP_i} \max_i m_2(H[P_i]),
\end{align*}
where we minimise over all possible partitions of the vertex set of $H$ into $r$ parts $P_1,\dots,P_r$.
Then the following holds.

\begin{thm}[Unbiased MB $H$-game]
\label{thm:H-game_Maker}
	Let $\gamma>0$, $r \ge 2$ be an integer, $\alpha \in ( \frac{r-2}{r-1},\frac{r-1}{r} ]$, and let $H$ be a fixed graph with $m_2^{(r)}(H)>0$.
	Further, let $G_\alpha$ be a graph on $n$ vertices with $\delta(G_{\alpha}) \geq \alpha n$, and let
	$p \ge n^{-1/m_2^{(r)}(H)+\gamma}$.
	Then a.a.s.~the following holds: 
	playing a $(1:1)$ Maker-Breaker game  on the edges 	
	of $G_{\alpha}\cup G_{n,p}$, Maker 
	has a strategy to obtain a copy of $H$.
\end{thm}

Note that $m_2^{(r)}(H)>0$ if and only if $H$ has chromatic number more than $r$ and, therefore, Theorem~\ref{thm:hgame0} covers the case $m_2^{(r)}(H)=0$ with $p=0$.
For many cases of $H$ Theorem~\ref{thm:H-game_Maker} again is optimal up to the constant $\gamma$ in the exponent.
But, similarly to the results for the $H$-game on $G_{n,p}$,
we do not have optimality when the relevant density
$m_2^{(r)}(H)$ is determined by a subgraph of $H$ isomorphic to $K_3$.
We will discuss more details in Section~\ref{sec:conc} together with some open problems.

\subsection{Our results - Waiter-Client games on randomly perturbed graphs}

Some of the approaches of our proofs for Maker-Breaker games can be modified to work for Waiter-Client games.
A $(1:b)$ Waiter-Client game (also referred to as $b$-{\em biased Waiter-Client game}) on some hypergraph $(X,\mathcal{F})$ is played as follows.
In every round,
Waiter chooses $b+1$ elements of $X$ that have not been chosen before (except for maybe the last round where Waiter could pick less elements), and she offers those to Client.
Client then claims one of these offered elements
(except for maybe the last round in the case when there is only one element left), while all the other elements go to Waiter.
The game is said to be Waiter's win if Client fully claims an element of $\mathcal{F}$; otherwise Client wins. 
This time, increasing the bias $b$ is never a disadvantage for Client, and hence there must be a 
{\em threshold bias} $b_{\mathcal{F}}$ such that
Waiter wins if and only if $b<b_{\mathcal{F}}$ (see e.g.~\cite{Tan2017}).
For the structures discussed above, we prove the following results.

\begin{thm}[Biased WC Hamiltonicity game]
	\label{thm:Ham_WC}
	For every real $\alpha>0$ there exist constants $c,C>0$ such that the following holds for large enough integers $n$. Let $G_{\alpha}$ be a graph on $n$ vertices with 
	$\delta(G_{\alpha})\geq \alpha n$, let 
	$b\leq cn$ be an integer, and let $p \geq \frac{Cb}{n}$. Then a.a.s.~the following holds: playing a $(1:b)$ Waiter-Client game on the edges of $G_{\alpha}\cup G_{n,p}$, Waiter has a strategy to force Client to occupy a Hamilton cycle.
\end{thm}

\begin{thm}[Biased WC $k$-vertex-connectivity game]
	\label{thm:vertex-connectivity_WC}
	For every real $\alpha>0$ and every integer $k$ there exist constants $C,c>0$ such that the following holds for large enough integers $n$. Let $G_{\alpha}$ be a graph on $n$ vertices with $\delta(G_{\alpha})\geq \alpha n$, let $b\leq cn$ be an integer, and let $p\geq \frac{Cb}{n^2}$. Then with probability at least $1-\exp(-c pn^2)$ the following holds: 
	playing a $(1:b)$ Waiter-Client game on the edges of $G_{\alpha}\cup G_{n,p}$, Waiter has a strategy to force Client to claim a spanning $k$-vertex-connected graph.
\end{thm}

\begin{thm}[Unbiased WC $H$-game]\label{thm:H-game_WC}
	For any integer $r \ge 2$ let $\alpha \in ( \frac{r-2}{r-1},\frac{r-1}{r} ]$, and let $H$ be a fixed graph. Further, let $G_\alpha$ be a graph on $n$ vertices with $\delta(G_{\alpha}) \geq \alpha n$, and let
	$p = \omega (n^{-1/m^{(r)}(H)})$.
	Then a.a.s.~the following holds:
	playing a $(1:1)$ Waiter-Client game on $G_\alpha \cup G_{n,p}$, Waiter has a strategy to force Client to claim a copy of $H$.
\end{thm}

Regarding the edge probability $p$,
note that the results for the Hamiltonicity game
and the $k$-vertex-connectivity game coincide
with the theorems from Maker-Breaker games;
the optimality can be explained analogously.
However, in contrast to Maker-Breaker games,
both Theorem~\ref{thm:Ham_WC} and Theorem~\ref{thm:vertex-connectivity_WC}
allow the bias to be linear in $n$.
For the corresponding games on $K_n$ this was already 
shown in~\cite{BHKL2016}. Although we use tools 
for the proof that are similar to the 
discussion of the corresponding Maker-Breaker
games, the difference regarding the bias $b$
requires some new ideas.
Moreover, for the $H$-game the threshold on $p$ differs from
the discussion of Maker-Breaker games. The bound 
in Theorem~\ref{thm:H-game_WC} is
optimal, since for $p = o (n^{-1/m^{(r)}(H)})$ it is known that a.a.s.~$G_\alpha \cup G_{n,p}$ does not contain a copy of $H$~\cite{KST_SmoothedAnalysis}.

\medskip

\subsection*{Organization of the paper.} 
In Section~\ref{sec:prel} we will summarise some
useful tools on probability, regularity and games.
In Section~\ref{sec:ham} we will consider the Hamiltonicity game and prove 
Theorem~\ref{thm:Ham}
and Theorem~\ref{thm:Ham_WC}. Section~\ref{sec:conn}
is devoted to the $k$-vertex-connectivity game, 
where we show Theorem~\ref{thm:vertex-connectivity}
and Theorem~\ref{thm:vertex-connectivity_WC}.
In Section~\ref{sec:hgame} we continue with the $H$-game.
We first start with the discussion of cycle games.
Then we use Subsection~\ref{subsec:maintain-regularity}
and Subsection~\ref{subsec:many-H} to prepare our strategy
for the general $H$-game.
Theorem~\ref{thm:H-game_Maker}
and Theorem~\ref{thm:H-game_WC} are proven afterwards in
Subsection~\ref{subsec:H-game-strategy}.
Note that Theorem~\ref{thm:clique_Maker} is a corollary
of Theorem~\ref{thm:H-game_Maker},
and Theorem~\ref{thm:hgame0} follows analogously to the proof of Theorem~\ref{thm:H-game_Maker} (see Remark~\ref{rem:r-chromatic}).
We will finish with a few concluding remarks and open problems in Section~\ref{sec:conc}.

\subsection*{Notation.} 

We use standard graph-theoretic notation, which closely follows~\cite{W2001}. 
In most of the proofs we will first describe Maker's or Waiter's strategy and afterwards discuss why it is possible for the respective player to follow this strategy.
We implicitly assume that Maker or Waiter forfeits the game when it is not possible to follow her strategy, while it will follow from the discussion that this does not happen.
Additionally, when some Maker-Breaker game is in progress, we emphasise the following. 
We let $\maker$ and $\breaker$ denote the graphs consisting of Maker's edges and Breaker's edges, respectively.
Any edge belonging to $\maker \cup \breaker$ is said to be claimed, while all the other edges in play are called free.
Analogously, for a Waiter-Client game let $\waiter$ and $\client$ denote Waiter's and Client's graphs, respectively.
  
 \bigskip


\section{Preliminaries}
\label{sec:prel}
  
\subsection{Probabilistic tools}

We will extensively use a plethora of tools which come from probability theory or which can be proven by some probabilistic argument. The first tool, which we often use to show that some non-negative random variable very unlikely exceeds a given bound, is Markov's inequality:

\begin{lemma}[Markov's inequality, see e.g.~\cite{JLR2000}]\label{lem:markov}
Let $X\geq 0$ be a random variable. For every $t\geq 0$ it holds that
\[
\Prob\left( X\geq t \right)
\leq \frac{\Exp(X)}{t}~ .
\]
\end{lemma}

We also often use Chernoff's inequalities (see e.g.~\cite{JLR2000}) to show that a binomial random variable $X \sim \operatorname{Bin}(n,p)$ is concentrated around its expectation $\mathbb{E}(X)=np$, where $n$ is the number of independent rounds and $p$ is the success probability.

\begin{lemma}\label{lem:Chernoff1}
If $X \sim \operatorname{Bin}(n,p)$, then
\begin{itemize}
    \item $\Prob(X<(1-\delta)np)< \exp\left(-\frac{\delta^2np}{2}\right)$ for every $\delta>0$, and
    \item $\Prob(X>(1+\delta)np)< \exp\left(-\frac{np}{3}\right)$ for every $\delta\geq 1$.
\end{itemize}
\end{lemma}

\begin{lemma}\label{lem:Chernoff2}
If $X \sim \operatorname{Bin}(n,p)$ and $k\geq 7 \Exp(X)$, then
\[\Prob(X\geq k)\leq \exp\left(-k\right)\, .\]
\end{lemma}

We use the following lemma (see e.g.~\cite{Hoeffding1963}) to bound the value of random variables that are distributed according to the hypergeometric distribution $\operatorname{Hypergeometric}(N, K, n)$, where from $N$ objects, of which $K$ are considered a success, $n$ are drawn without replacement.

\begin{lemma}\label{lem:hypergeometric}
	If $X \sim \operatorname{Hypergeometric}(N, K, n)$ and $t > 0$, then
	\[\Prob \left(X < \left(\frac{K}{N} - t\right)n\right) < \exp\left(-2t^2n\right)\, .\]
\end{lemma}

Next we focus on random graphs. The following lemma provides some useful properties
which are very likely to hold in a 
random graph $G_{n,p}$. The proof
is a standard applications of the above concentration bounds.

\begin{lemma}\label{lem:gnp2}
Let $\beta>0$ and let $p=p(n)\in (0,1)$ such that $pn\rightarrow \infty$. If we generate a random graph $G \sim G_{n,p}$ on the vertex set $[n]$, then a.a.s. the following properties hold:
\begin{enumerate}
\item $e_G(A,B) \geq 0.5 p|A||B|$ for every pair of disjoint sets $A,B\subset [n]$ of size 
$|A|,|B|\geq \beta n$,
\item $d_G(v)\leq \max\{2np, 30 \log n\}$ for every $v\in V(G)$,
\item $e(G)\leq n^2p$. 
\end{enumerate}

\end{lemma}

\begin{proof}
We prove (1) first. For any specific pair of sets $A,B$ of size $|A|=|B|=\beta n$ it holds that $\mathbb E (e_G(A,B)) = p |A||B|$, so applying Lemma~\ref{lem:Chernoff1} we obtain
\[\mathbb P \left(e_G(A,B) < 0.5 p (\beta n)^2\right) < \exp\left(-\frac18 p(\beta n)^2 \right)\, .\]
Taking union bound over all possible choices of $A$ the claim follows.

Now, to prove (2), observe that $\mathbb E (d_G(v)) = (n-1) p$ for any $v \in V$. If $(n-1)p\geq 4\ln n$, then 
Lemma~\ref{lem:Chernoff1} yields 
\[\mathbb P \left(d_G(v) > 2 np \right) <\exp \left( -\frac{(n-1)p}{3} \right)
<\exp \left( -\frac{4}{3} \ln n \right)\, .\]
If otherwise $(n-1)p \leq 4 \ln n$, then
Lemma~\ref{lem:Chernoff2} gives 
\[\mathbb P \left(d_G(v) > 30\ln n \right) < \exp \left( - 30\ln n \right) \, .\]
In either case, taking union bound over all possible choices of $v$ the claim follows.

To prove (3) notice that $\mathbb E (e(G)) = {n \choose 2} p < \frac{n^2p}{2}$. Applying Lemma~\ref{lem:Chernoff1} once again shows, that
\[\mathbb{P} \left(e(G) > n^2 p \right) < \exp \left( - \frac13 n^2p \right)=o(1)\,.\]
This proves the lemma.
\end{proof}

Lastly, by using a probabilistic argument, 
we prove two lemmas that help us find two
useful partitions of a graph with large minimum degree or
large connectivity, respectively.

\begin{lemma}\label{lem:vertex-partition}
		Let $\alpha > 0$, $k\geq 1$ be an integer, and 
		$G$ be a graph on $n$ vertices with 
		$\delta(G) \geq \alpha n$. 
		If $n$ is large enough, 
		there exists a partition 
		$V(G) = U_1 \cup U_2 \cup ... \cup U_k$ with
		\begin{enumerate}
			\item[(1)] $|U_i| = \lfloor \frac{n}{k} \rfloor$ or $|U_i| = \lceil \frac{n}{k} \rceil$, for all $i \in [k]$,
			\item[(2)] $e_{G}(v, U_i) \geq \frac{\alpha}{2}|U_i|$ for all $v \in V(G)$ and $i \in [k].$
		\end{enumerate}
\end{lemma}
	
\begin{proof}
	We show that a random partition of the vertices a.a.s.~fulfils the properties, thus proving the statement for large $n$. 
	We choose a partition of $V(G)$ into $k$ sets $U_1, U_2, ..., U_k$ such that
	\[
	\left\lfloor \frac{n}{k} \right\rfloor \leq |U_1| \leq |U_2| \leq \ldots \leq |U_k| \leq \left\lceil \frac{n}{k} \right\rceil
	\]
	uniformly at random among all such partitions.  
	For any $v \in V(G)$ and $i\in [k]$, we thus have 
\[\Prob\left(e_G(v, U_i) = \ell \right) = \frac{{{d_G(v)} \choose {\ell}}{{n - d_G(v)} \choose {|U_i| - \ell}}}{{{n} \choose {|U_i|}}}\, .\]
Therefore, $e_G(v, U_i) \sim \operatorname{Hypergeometric}\left(n, d_G(v), |U_i|\right)$, 
and by Lemma~\ref{lem:hypergeometric} we have 
\[\Prob\left(e_G\left(v, U_i\right) < \frac{\alpha}{2}|U_i|\right) < \Prob\left(e_G(v, U_i) < \left(\frac{d_G(v)}{n} - \frac{\alpha}{2}\right)|U_i|\right) < \exp\left( - \frac{1}{2} \alpha^2 |U_i| \right)\,.\] 
Taking union bound over all possible choices of $v \in V(G)$ and $i \in [k]$, the claim follows.
\end{proof}

\begin{lemma}\label{lem:partition_connected}
For every $\beta \in (0,1)$ there exists a constant $\gamma>0$
such that the following holds for every large enough $n$:
let $G$ be any $\beta n$-vertex-connected graph on at most $n$ vertices.
Then there exists a partition $G=G^1\cup G^2$
such that both parts are
$\gamma n$-vertex-connected on $V(G)$.
\end{lemma}

\begin{proof}
Set $\gamma = \beta \left(\frac{1}{2}\right)^{4/\beta + 2}$.
Let $u,v\in V(G)$ be different vertices. 
By Menger's Theorem (see e.g. Theorem 4.2.17 in~\cite{W2001}) there exist $\beta n$
internally vertex-disjoint paths
between $u$ and $v$. At least
$\frac{\beta n}{2}$ of these paths have length
at most $\frac{4}{\beta}$. From now on, fix a family of 
such $\frac{\beta n}{2}$ paths and denote it with $P_{\{u,v\}}$.

In the following we will consider a random partition $G=G^1\cup G^2$, where each edge of $G$
is either added to $G^1$ or $G^2$ uniformly at random
with probability $\frac{1}{2}$ and independently of all other choices. It is enough to show that a.a.s.~such a partition
has the desired property.

Let $i\in [2]$. A path from $P_{\{u,v\}}$ is contained in
$G^i$ with probability at least
$\left(\frac{1}{2}\right)^{4/\beta}$. Hence, the random variable 
$X^i_{\{u,v\}}$, which counts the number of such paths
landing in $G^i$,
stochastically dominates 
$\operatorname{Bin}\left(|P_{\{u,v\}}|,\left(\frac{1}{2}\right)^{4/\beta}\right)$.
By an application of Chernoff (Lemma~\ref{lem:Chernoff1})
and the union bound, it follows a.a.s.~that
for every $i\in [2]$ and every distinct $u,v\in V(G)$,
\[
X^i_{\{u,v\}} \geq \frac{1}{2}|P_{\{u,v\}}| \cdot \left(\frac{1}{2}\right)^{4/\beta} = \gamma n~ .
\]
But this means that a.a.s.~for both $i\in [2]$,
$G^i$ has at least $\gamma n$ internally vertex-disjoint paths between any pair of vertices, and hence is
$\gamma n$-vertex-connected.
\end{proof}

\subsection{Regularity tools}

We will use standard tools from regularity theory in combination with other techniques as used by Das and Treglown~\cite{DT_PeturbedRamsey} for their perturbed Ramsey results.
First, we introduce the relevant terminology. Let a graph $G$ be given.
The \emph{density} of a pair $(A,B)$ of disjoint subsets of $V(G)$ is defined as $d_G(A,B)=\frac{e_G(A,B)}{|A| |B|}$.
For $\eps>0$ we say that $(A,B)_G$ is \emph{$\eps$-regular}, if for all subsets $X \subseteq A$ and $Y \subseteq B$ with $|X| \ge \eps |A|$ and $|Y| \ge \eps |B|$ we have $|d_G(A,B)-d_G(X,Y)|\le \eps$.
This notion allows us to control the distribution of the edges between $A$ and $B$.
When passing to smaller sets this property is preserved with adjusted parameters.

\begin{lemma}[Regularity slicing, Fact 1.5~in~\cite{KS_RegularityApplications}]
	\label{lem:slicing}
	Let $(A,B)_G$ be an $\eps$-regular pair of density at least $\delta$ and let $\eta > \eps$.
	For any $X \subseteq A$ and $Y \subseteq B$ of size at least $\eta|A|$ and $\eta |B|$ the pair $(X,Y)_G$ is $\eps'$-regular with density at least $\delta'$, where $\eps'= \max \left\{ \frac{\eps}{\eta} ,2 \eps \right\}$ and $|\delta'-\delta| <  \eps$.
\end{lemma}

We will further use the following consequence of the classical Regularity Lemma~\cite{KS_RegularityApplications} together with Tur\'{a}n's theorem, which allows us to find a family of sets such that all pairs are regular.

\begin{lemma}[Corollary~2.4 in~\cite{DT_PeturbedRamsey}]\label{lem:reg-turan}
	For $0 < 3 \eps \le \delta <1$  there exists an $\eta>0$ and $n_0$ such that for all $n \ge n_0$ and $r \ge 2$ the following holds.
	For any $n$ vertex graph $G$ of density at least $\frac{r-2}{r-1}+ \delta$ there are pairwise disjoint sets $V_1,\dots,V_r \subseteq V(G)$ such that $|V_i| \ge \eta n$ for $1 \le i \le r$, and $(V_i,V_j)_G$ is $\eps$-regular with density at least $\frac{\delta}{2}$ for $1 \le i < j \le r$.
\end{lemma}

Once the regular pairs are in place we want to use them to combine structures that we can find within each of the sets.
For this we ideally want large sets such that all small subsets have many common neighbours.
The next lemma follows from the dependent random choice technique~\cite{FS_DRC}.

\begin{lemma}[Lemma~2.5 in~\cite{DT_PeturbedRamsey}]\label{lem:reg_drc}
	Given $r \ge 2$, $\delta,\beta>0$, $\ell \in \mathbb{N}$ there exists a $\nu>0$ such that for $0<\eps<\min \{ \frac{1}{2r}, \frac{\delta}{2} \}$ there exists an $m_0$ such that the following holds for any $m \ge m_0$.
	Let $V_1,\dots,V_r$ be disjoint sets of vertices of size $m$ from a graph $G$ such that $(V_1,V_i)_G$ is $\eps$-regular with density at least $\delta$ for $2 \le i \le r$.
	Then there exists a subset $U \subseteq V_1$ of size at least $m^{1-\beta}$ such that any $\ell$ vertices from $U$ have at least $\nu m$ common neighbours in each of the sets $V_i$ for $2 \le i \le r$.
\end{lemma}

\subsection{Positional games tools}

Helpful tools for the study of positional games are potential functions. One of the central results in this area that was proven by the use of such a tool, is the following theorem, usually referred to as Beck's Criterion. Note that in a $(p:q)$ Maker-Breaker game, in each round Maker and Breaker claim $p$ and $q$ elements of the board, respectively.

\begin{thm}[Beck's Criterion, Theorem~1 in~\cite{Beck1982}]\label{thm:beck_criterion}
Let $(X,\mathcal{F})$ be a hypergraph satisfying
$$
\sum_{F\in \mathcal{F}} (1+q)^{-|F|/p + 1} < 1
$$
then Breaker has a strategy to win the $(p:q)$ Maker-Breaker game on $(X,\mathcal{F})$, independent of who starts the game.
\end{thm}

Although the result guarantees a winning strategy for Breaker
it can also be beneficial for the description of strategies for Maker. Here the idea is to redefine a Maker-Breaker game in such a way that the roles of both players are switched.
Such an approach is fairly standard and has been applied multiple times in the literature (see e.g.~\cite{BeckBook,HKSS_PosGames}). In order to avoid repetitions of such an argument, we provide a reformulation of Beck's Criterion which can be applied directly when we study Maker strategies for $(1:b)$ games throughout the paper.
Given a hypergraph $(X,\mathcal{F})$ we define the transversal hypergraph $(X,\mathcal{F}^\ast)$ as follows:
\[
\mathcal{F}^\ast := 
\left\{
S\subset X:~ (\forall F\in \mathcal{F}:~ S\cap F\neq \varnothing)
\right\}~ .
\]

\begin{cor}
\label{cor:beck_transversal}
Let $(X,\mathcal{F})$ be a hypergraph satisfying
$$
\sum_{F\in \mathcal{F}} 2^{-|F|/b + 1} < 1
$$
then Maker has a strategy to win the $(1:b)$ Maker-Breaker game on $(X,\mathcal{F}^\ast)$, independent of who starts the game.
\end{cor}

\begin{proof}
Maker wins the game on $(X,\mathcal{F}^\ast)$ if and only
if she claims all elements of a transversal $F\in \mathcal{F}^\ast$. That is, she needs to prevent Breaker from claiming a set $F\in\mathcal{F}$ completely. Hence Maker considers playing as $\mathcal{F}$-Breaker, playing with bias 1.
By Beck's Criterion she has a strategy for this when   
$
\sum_{F\in \mathcal{F}} 2^{-|F|/b + 1} < 1.
$
\end{proof}

A result similar to Beck's Criterion, but for \emph{Client-Waiter games}, has been proven by Bednarska-Bzd\c{e}ga~\cite{B2013},
and a similar reformulation yields the following
criterion for Waiter-Client games.

\begin{thm}[Theorem~2.2 in~\cite{BHKL2016}]
\label{thm:WC_transversal}
Let $(X,\mathcal{F})$ be a hypergraph satisfying
$$
\sum_{F\in \mathcal{F}} 2^{-|F|/(2b-1) + 1} < 1
$$
then Waiter has a strategy to win the $(1:b)$ Waiter-Client game on $(X,\mathcal{F}^\ast)$, independent of who starts the game.
\end{thm}

Next to the above winning criteria we will also make use of the {\em trick of fake moves} which roughly says that
Maker's situation does not worsen when
Breaker's bias decreases. While the above criteria do not immediately provide fast strategies, the following can be used to obtain suitable upper bounds on the number of rounds needed to claim a winning set.

\begin{lemma}[Trick of fake moves \cite{BeckBook}]\label{lem:fake_moves_MB}
Let $(X,\mathcal{F})$ be a hypergraph. Let $b$ be a positive integer. 
If Maker has a winning strategy for
the $(1 : b)$ Maker-Breaker game on $(X, {\mathcal F})$, then she also has a strategy to
win the game within $\left\lceil \frac{|X|}{b+1} \right\rceil$ 
rounds even when in each round Breaker is allowed
to claim between $0$ and $b$ free elements.
\end{lemma}

The idea of the above lemma is straightforward:
in every move of the new game on $(X, {\mathcal F})$, 
let Maker (in her mind) give as many additional elements of $X$ to Breaker as necessary, such that their number together with Breaker's elements sums up to $b$. Whenever, in a later round, Breaker claims one of these additional elements, Maker (in her mind) gives another free element to Breaker.
Then, applying the winning strategy
from the $(1:b)$ game, the result follows.

\smallskip

Along the lines of the above argument, the following statement for Waiter-Client games can be proven.

\begin{lemma}[Trick of fake moves \cite{BeckBook}]\label{lem:fake_moves_WC}
Let $(X,\mathcal{F})$ be a hypergraph. Let $b$ be a positive integer. 
If Waiter has a winning strategy for
the $(1 : b)$ Waiter-Client game on $(X, {\mathcal F})$, then she also has a strategy to
win the game within $\left\lceil \frac{|X|}{b+1} \right\rceil$ 
rounds even when in each round she is allowed
to offer between $2$ and $b+1$ free elements.
\end{lemma}

Moreover, we will make use of an argument which ensures that Maker
can claim a spanning graph of a suitable minimum degree.
In order to do so, we may use an auxiliary game, called MinBox$(n,D,\gamma,b)$, which was introduced in~\cite{FKN2015}.
This $(1:b)$ Maker-Breaker game is played on a family of $n$ disjoint sets (called \textit{boxes}), each of size at least $D$, with the board $X$ being the union of all boxes.
Maker is called the winner if she manages
to occupy at least $\gamma|B|$ elements from each box $B$. The following theorem gives a criterion
for Maker to win the game.

\begin{thm}[Theorem 2.5 in~\cite{FKN2015}]\label{thm:boxgame}
Let $n,D,b$ be positive integers and let $0<\gamma<1$.
If $\gamma<\frac{1}{b+1}$ and $D>\frac{b(\ln n +1)}{1-\gamma(b+1)}$,
then Maker wins the game MinBox$(n,D,\gamma,b)$.
\end{thm}

 \bigskip


\section{Hamiltonicity game}
\label{sec:ham}

We will prove Theorem~\ref{thm:Ham} and Theorem~\ref{thm:Ham_WC} in this section.
Depending on the size of the bias $b$, we will use different arguments.
For Maker-Breaker games, the argument in Subsection~\ref{sec:Ham_bsmall} goes through for any $b=o(\sqrt{n})$ while the argument in Subsection~\ref{sec:Ham_blarge} works for any $b=\Omega(\log n)$.
Both cannot be significantly improved, because Theorem~\ref{thm:Ham2} limits the first approach and to improve the second, we would need better estimates for the number of expander graphs with a linear number of edges.
Because of the overlap, we will use the restriction $b\leq n^{0.49}$ in Subsection~\ref{sec:Ham_bsmall}.
Theorem~\ref{thm:Ham_WC}
is proven in Subsection~\ref{subsec:Ham_WC}.

\subsection{Maker-Breaker Hamiltonicity game with small edge probabilities and biases}\label{sec:Ham_bsmall}

In the following we will prove Theorem~\ref{thm:Ham}
for $b\leq n^{0.49}$. In order to do so,
we will make use of the following sufficient condition
for a graph $G$ to have a Hamilton cycle.

\begin{thm}[Theorem 2.5 in~\cite{HKS2009}]\label{thm:Ham2}
Let $12\leq d\leq \sqrt{n}$ and let $G$ be a graph on $n$ vertices
such that the following properties hold:
\begin{enumerate}
\item[(i)] $|N_G(S)|\geq d|S|$ for every $S\subset V(G)$ of size $|S|\leq \frac{n \log d}{d\log n}$,
\item[(ii)] $e_G(A,B)>0$ for every pair of disjoint sets $A,B\subset V(G)$ of size
$|A|,|B|\geq \frac{n \log d}{1035 \log n}$.
\end{enumerate}
Then, provided that $n$ is large enough, $G$ contains a Hamilton cycle.
\end{thm}

\begin{proof}[Proof of Theorem~\ref{thm:Ham} for $b\leq n^{0.49}$]

Let $\alpha$ be given. Choose $\delta=10^{-4}\alpha$, 
$C=10^8\delta^{-2}$ and $c = 10^{-3}\delta^{2}\alpha$.
Let
$1 \le b \le n^{0.49}$ and, by monotonicity,
let us assume that
$p = \frac{Cb}{n}$.
Let $G_{\alpha}$ be any $n$-vertex graph with $n$ large enough and minimum degree $\delta(G_{\alpha}) \ge \alpha n$.
	Next, we reveal $G_2 \sim  \gnp$ on $V(G_{\alpha})$.
For the remainder of the proof, let us condition on $G_2$ having the properties from Lemma~\ref{lem:gnp2} with $\beta= \frac{\delta}{1035}$, and
let $G_1 = G_\alpha \setminus G_2 $.
Then
$\delta(G_1) \geq \frac{\alpha n}{2}$. 

Given the properties from Lemma~\ref{lem:gnp2}, 
we will show that
Maker has a strategy to occupy a Hamilton cycle 
in the $(1:b)$ Maker-Breaker game on 
$G_1\cup G_2$. 
In order to do so, Maker ensures that 
her final graph $\maker$ will satisfy the following two properties:
\begin{enumerate}
\item[(1)] $|N_{\maker}(S)|\geq n^{\delta} |S|$ for every $S\subset V(G)$ of size $|S|\leq \delta n^{1-\delta}$,
\item[(2)] $e_{\maker}(A,B)>0$ for every pair of disjoint sets $A,B\subset V(G)$ of size
$|A|,|B|\geq \beta n$.
\end{enumerate}
Then, according to Theorem~\ref{thm:Ham2} (applied with $d=n^{\delta}$), it follows that $\maker$ contains a Hamilton cycle, provided $n$ is large enough.

Before we describe Maker's strategy,
let us fix a partition
$
V(G_1)=V_1\cup V_2\cup V_3\cup V_4
$
such that for each $i\in [4]$ we have 
$\frac{n}{4}-1 \leq |V_i|\leq \frac{n}{4} + 1$
and for every $v\in V_i$ and every $j\neq i$
we have
$$d_{G_1}(v,V_j)>\frac{\alpha}{5}|V_j|>\frac{\alpha}{25}n~ .$$ 
The existence of such a partition is given by Lemma~\ref{lem:vertex-partition}.

\medskip

{\bf Strategy description:} 
Consider the edge sets of $G_1$ and $G_2$ as 
two edge disjoint boards, on which Maker plays alternately.
Then Breaker claims at most $2b$ edges between two moves of Maker on the same board, and hence
by the trick of fake moves (Lemma~\ref{lem:fake_moves_MB}), we may assume that on each of the boards 
a $(1:2b)$ Maker-Breaker game is played.

On $G_1$ Maker plays in such a way 
that she obtains a subgraph of $G_1$ 
with Property (1);
on $G_2$ Maker plays in such a way 
that she obtains a subgraph of $G_2$ 
with Property (2). All details will be given 
later in the strategy discussion.

\medskip

{\bf Strategy discussion:} 
Notice that, if Maker can follow her strategy, her final graph will contain a Hamilton cycle. Hence, it remains to be shown that she can indeed follow her strategy and reach the different goals on $G_1$ and $G_2$. We consider each of the boards $E(G_1)$ and $E(G_2)$ separately.

\medskip

\textit{$(1:2b)$ game on $E(G_1)$:}
We split $G_1$ into the edge-disjoint subgraphs
$$
G_{1,1} = G_1[V_1,V_2] \cup G_1[V_2,V_3] \cup G_1[V_3,V_4]
 \cup G_1[V_4,V_1]
 ~~ \text{ and } ~~
G_{1,2} = G_1[V_1,V_3] \cup G_1[V_2,V_4]~ .
$$

Maker plays on each of the boards $E(G_{1,1})$
and $E(G_{1,2})$ alternately.
By the trick of fake moves (Lemma~\ref{lem:fake_moves_MB}), we may thus assume
that a $(1:4b)$ is played on each of the boards.

\smallskip

On $G_{1,1}$ Maker ensures that
every vertex in her graph will have degree at least 
$\frac{\alpha n}{200b}$. She can do this as follows:
she considers playing the game
MinBox$(n,D,\gamma,4b)$ with the $n$ disjoint 
boxes $E_{G_1}(v,V_{i+1})$, for $i\in [4]$ and $v\in V_i$
(where we set $V_5=V_1$), where each box has size at least 
$D=\frac{\alpha}{25}n$, and where we set $\gamma =\frac{1}{8b}$.
By the choice of all parameters,
one verifies that
$$
\gamma < \frac{1}{4b+1}~ \text{ and } \frac{4b(\ln n +1)}{1-\gamma(4b+1)}<n^{0.5}<D
$$
provided $n$ is large enough.
Hence, applying Theorem~\ref{thm:boxgame}, Maker wins this auxiliary game. That is, Maker claims at least
$\gamma D = \frac{\alpha n}{200b}$ elements of every box.

\smallskip

On $G_{1,2}$, Maker ensures to claim an edge 
in each edge set contained in
\begin{align}\label{Ham:F1set}
\mathcal{F}_1 &
	= \left\{ 
	E_{G_1}(A,B):~ 
	\begin{array}{l}	
	\exists i,j\in [4]~ \text{with} ~
		|i-j|=2,\\ 
		A\subset V_i,~ B\subset V_j,~
		|A|=n^{0.5}~ \text{and} ~
		|B|=\left(1-\frac{\alpha}{10}\right)|V_j|
	\end{array}	
	\right\}~ .
\end{align}
For this, she considers the auxiliary $(1:4b)$ Maker-Breaker game with winning sets $\mathcal{F}_1^\ast$.
In order to show that Maker wins this game,
it will be enough to show that
$$
\sum_{F\in \mathcal{F}_1} 2^{-|F|/(4b) + 1}
= o(1),
$$
according to Corollary~\ref{cor:beck_transversal}.
The latter holds by the following reason:
Since $d_{G_1}(v,V_{j})>\frac{\alpha}{5}|V_j|$
for every $v\in V_i$ and every $j\neq i$, 
we see that for every $F=E_{G_1}(A,B)\in \mathcal{F}_1$ it holds that
$
|F|\geq |A|\cdot \frac{\alpha}{10}|V_j| > n^{1.495}~ .
$
Hence, since $b\leq n^{0.49}$,
\begin{equation}\label{Ham:F1estimate}
\sum_{F\in \mathcal{F}_1} 2^{-|F|/(4b) + 1}
<
(2^n)^2\cdot 2^{-n^{1.005}/4 + 1}
= o(1)~ .
\end{equation}

It remains to show that, 
because of the achievements from the games on $G_{1,1}$
and $G_{1,2}$, Maker is able to create a graph with Property~(1).
For this, let $S$ be any subset of $V(G_1)$ of size 
$|S|\leq \delta n^{1-\delta}$.
Our goal is to show that
$|N_{\maker}(S)|\geq n^{\delta}|S|$
holds by the end of the game.
If $|S| < n^{0.505}$, then we immediately have 
$$|N_{\maker}(S)| \geq \delta(\maker) - |S| \geq \frac{\alpha n}{200 b} - |S| > n^{\delta} |S|$$
for large $n$.
Otherwise, it holds that 
$n^{0.505}\leq |S| \leq \delta n^{1-\delta}$.
In this case we can find 
a subset
$S' \subset S\cap V_i$ of size $|S'| = n^{0.5}$
for some $i\in [4]$.
If $|N_{\maker}(S')| \geq 2 \delta n$ holds, then we immediately
get
$$|N_{\maker}(S)| \geq |N_{\maker}(S')| - |S| \geq 2 \delta n - \delta n^{1-\delta} > \delta n \geq n^{\delta} |S|$$
and we are done. So, assume for a contradiction, that
$|N_{\maker}(S')| < 2 \delta n$. 
Let $j\in [4]$ with $|i-j|=2$.
Then we are able to find a set $B\subset V_j$ of size 
$|B| = |V_j|- 2\delta n$ 
with $e_{\maker}(S',B) = 0$. However, 
we have  $|B| > \left(1 - \frac{\alpha}{10} \right)|V_j|$
by the choice of $\delta$, and hence, by the result from the game on $G_{1,2}$, Maker needs to have an edge
in $E_{G_1}(S',B)$.
This gives the desired contradiction.

\medskip

\textit{$(1:2b)$ game on $E(G_2)$:}
Maker's goal is to occupy a spanning subgraph of $G_2$
which fulfils Property (2).
That is, she aims to claim an element from each of the sets contained in
\begin{equation}\label{Ham:F2set}
\mathcal{F}_2 := \left\{ E_{G_2}(A,B): A,B \subset V(G_2) \text{ disjoint and } |A|=|B|= \beta n \right\}.
 \end{equation}
By the properties from Lemma~\ref{lem:gnp2}, we have
$e_{G_2}(A,B) \geq \frac12p|A||B|$ for every pair of disjoint sets $A,B \subset V(G_2)$ of size at least $\beta n$,
and hence $|F| \geq \frac12 p (\beta n)^2$ for any $F \in \mathcal{F}_2$. This yields
\begin{align}\label{Ham:F2estimate}
 \sum_{F \in \mathcal{F}_2} 2^{-|F|/(2b)+1} \leq 2^{2n} \cdot 2^{-p \beta^2 n^2 / (4b)  +1} = o(1) ,
\end{align}
where we use that $\frac{p\beta^2n^2}{4b}\geq \frac{C\beta^2 n}{4}>3n$ by the choice of $p$, $C$ and $\beta$.
Hence, 
following Corollary~\ref{cor:beck_transversal}, Maker reaches her goal
and is able to occupy a subgraph satisfying Property~(2). 
\end{proof}

\medskip

\subsection{Maker-Breaker Hamiltonicity game 
with large edge probabilities and biases}
\label{sec:Ham_blarge}

In the following we will prove Theorem~\ref{thm:Ham}
for $b\geq \ln n$.
To create a Hamilton cycle in this regime we will combine the approach by Krivelevich~\cite{Kriv_Ham} for the biased Hamilton cycle game on $K_n$ with ideas of~Ferber, Glebov, Krivelevich, and Naor~\cite{FGKN_BiasedGames} for biased games on random boards.
The following lemma follows from a close inspection of the proof from Gebauer and Szab\'{o}~\cite[Theorem~1.2]{GS_MinDegree} on the minimum degree game. 
A similar statement was observed by Krivelevich~\cite[Lemma~3]{Kriv_Ham} for the game on the complete graph.

At any moment during the game and for any vertex $v$, set 
$\dang(v):=d_{\breaker}(v)-2b d_{\maker}(v)$.
Consider the following strategy $\mathcal{S}$ for Maker:
As long as there exists a vertex $v$ of degree less than 16,
she chooses any such vertex which maximises $\dang(v)$,
and then she claims an arbitrary free edge incident to $v$.
The following statement holds.

\begin{lemma}
	\label{lem:MinDegree}
	For every $\alpha>0$ there exists a constant $c>0$ and an integer $n_0$ such that for every $n$-vertex graph $G$ with $n \ge n_0$ and minimum degree $\delta(G) \ge \alpha n$ the following holds for $b \le \frac{c n}{\log n}$.
	If in a $(1:b)$ Maker-Breaker game on $E(G)$ 
	Maker plays according to strategy $\mathcal{S}$,
	then Maker's degree of every vertex $v$ will be 16 
	at some point during the game, and right at the moment
	when this degree is reached it holds that
	$d_{\breaker}(v)\leq \frac{\alpha n}{2}$. 
\end{lemma}

We briefly sketch how the argument from~\cite[Theorem~5.3.6]{HKSS_PosGames} for $K_n$ can be adapted to $G$.

\begin{proof}[Sketch of the proof of Lemma~\ref{lem:MinDegree}]
Let $c=\frac{\alpha}{5}$. Without loss of generality, we assume that Breaker starts the game. Let Maker play according to strategy 
$\mathcal{S}$.
For any integer $i\geq 1$, denote with $\maker_i$ and $\breaker_i$
the $i^{\text{th}}$ move of Maker and Breaker, respectively,
and denote with $v_i$ the vertex chosen by Maker
in her $i^{\text{th}}$ move according to strategy $\mathcal{S}$.
For both $X\in \{\maker_i,\breaker_i\}$ and any vertex $v\in V(G)$ 
let $\dang_X(v)$ denote the danger value of $v$ immediately before $X$ happens. Moreover, for any subset $I\subset V(G)$ define
$$
\overline{\dang}_X(I):=\frac{\sum_{v\in I} \dang_X(v)}{|I|}
$$
to be the average danger in $I$ immediately before $X$ happens.
Assume Maker fails, i.e.~there happens to be a point of the game (say immediately after Breaker's $g^{\text{th}}$ move
for some $g\in \mathbb{N}$)
where there exists a vertex $v_g$ such that $d_{\maker}(v_g)<16$
but $d_{\breaker}(v_g)>\frac{\alpha n}{2}$. Then, in particular,
$$\dang_{\breaker_g}(v_g)> \frac{\alpha n}{2} - b - 2b\cdot 15 = \left( \frac{\alpha}{2} - o(1) \right)n.$$
Fix $g$ and $v_g$, and set $I_i=\{v_{g-i},\ldots,v_g\}$
for any $0\leq i\leq g-1$.
Then Corollary~3.3 from~\cite{GS_MinDegree}
still applies as its proof is independent of the graph which the game is played on. 
Accordingly, we can continue with the same calculations
that are given afterwards and based on~\cite[Corollary~3.3]{GS_MinDegree}.
In particular, we obtain
one of the following two estimates, with 
$k= \lceil \frac{n}{\ln n} \rceil$, $r=|I_g|$ and $H_s=\sum_{j=1}^{s} \frac{1}{j}$:
\begin{align*}
& \overline{\dang}_{\breaker_1}(I_{g-1}) \geq 
	\overline{\dang}_{\breaker_g}(I_{0}) - bH_k - k ~~~ \text{or}
& \overline{\dang}_{\breaker_1}(I_{g-1}) \geq 
	\overline{\dang}_{\breaker_g}(I_{0}) - b(2H_r - H_k) - k~ .
\end{align*}
In any case, since 
$\overline{\dang}_{\breaker_g}(I_{0}) = \dang_{\breaker_g}(v_g)$, 
we can conclude that
\[ \overline{\dang}_{\breaker_1}(I_{g-1}) \geq 
	\left( \frac{\alpha}{2} - 2c - o(1) \right)n > 0\]
for large enough $n$, where we use that $bH_n \leq \frac{cn}{\ln n}(\ln n +1) = (c+o(1))n$.
However, this is a contradiction,
since $\overline{\dang}_{\breaker_1}(I_{g-1})=0$
by its definition.
\end{proof}

Krivelevich~\cite[Lemma~4]{Kriv_Ham} observed that we can use the freedom of the choice in the strategy $\mathcal{S}$ to build an expander graph.

\begin{dfn}\label{def:expander}
Let $R\in\mathbb{N}$.
A graph $G$ is an $R$-expander if $|N_G(A)|\geq 2|A|$
for every $A\subset V(G)$ with $|A|\leq R$.
\end{dfn}

\begin{lemma}
	\label{lem:Expander}
	For every $\alpha>0$ there exist constants $c>0$, $\varepsilon>0$ and an integer $n_0$ such that for any $n$-vertex graph $G$ with $n \ge n_0$ and minimum degree $\delta(G) \ge \alpha n$ the following holds for $b \le \frac{c n }{\log n}$.
	In the $(1:b)$ game on $G$ Maker is able to build an $\varepsilon n$-expander in at most $16 n$ rounds.
\end{lemma}

We repeat the proof from~\cite{Kriv_Ham} adapted to our setting.

\begin{proof}[Proof of Lemma~\ref{lem:Expander}]
	Given $\alpha>0$ we let $0<\varepsilon < \alpha^8 2^{-81}$ and let $c>0$ and $n_0$ be given by Lemma~\ref{lem:MinDegree}.
	Let $G$ be a graph on $n$ vertices with $n \ge n_0$ and minimum degree $\delta(G) \ge \alpha n$.
	Further let $b \le \frac{c n }{\log n}$.
	
	Maker's strategy is to play according to strategy 
	$\mathcal{S}$ where she chooses the 'arbitrary edge'
	described in that strategy uniformly at random
	among all suitable edges.
	After at most $16n$ moves the game stops, since we stop 
	when Maker's graph $\maker$ has minimum degree $16$ and
	since in each of her moves Maker always increases
	the degree of a vertex of degree less than 16.
	Suppose that by now $\maker$ 
	is not an $\varepsilon n$-expander.
	This means, that there must be a set 
	$A \subseteq V(G)$ of size 
	$i \le \varepsilon n$ such that all neighbours of the set 
	$A$ with respect to Maker's graph are contained in a set 
	$B \subseteq V(G)$ of size $2i-1$.
	Since $\delta(\maker)\geq 16$, we can assume 
	$i \ge 5$ and 
	that there are at least $8i$ 
	edges between $A$ and $A\cup B$ in Maker's graph.
	
	In order to finish the proof, it will be enough to show 
	that with probability $o(1)$ there will be sets 
	$A$ and $B$
	as described above with at least $8i$ edges
	between $A$ and $A\cup B$.	
	
	Fix any sets $A$ and $B$ with 
	$|A|=i\in [5,\eps n]$ and $|B|=2i-1$.	
	First note that at most $16|A\cup B| < 48i$
	edges were chosen by dangerous vertices in $A\cup B$.
	For any edge $e=uv$ that was obtained after Maker chose 
	$v \in A \cup B$ as a dangerous vertex, consider Maker's 
	possible amount of choices for the vertex $u$. 
	Since $v$ was dangerous, it held that $d_{\maker}(v)<16$, and because of Lemma~\ref{lem:MinDegree} $d_{\breaker}(v) < \frac{\alpha n}{2}$, which means that	there were  
	$d_{G\setminus (\maker\cup \breaker)}(v) \geq 
	\frac{\alpha n}{2} - 16$ options 
	for the vertex $u$.
	Thus, the probability that $u$ also ended up in $A\cup B$ is at most 
	$\frac{3i-1}{\frac{\alpha n}{2}-16} < \frac{8i}{\alpha n}$,
	and therefore the probability that from the at most $48i$ 
	chosen edges at least $8i$ edges end up between
	$A$ and $A\cup B$ is bounded from above by
	$\binom{48i}{8i} \left(\frac{8i}{\alpha n}\right)^{8i}$.	
	
	In particular, 
	using a union bound over all choices of $i$, 
	$A$ and $B$, Maker's randomised strategy to create an 
	$\varepsilon n$-expander fails with probability at most
	\begin{align*}
	\sum_{i=5}^{\varepsilon n} \binom{n}{i} \binom{n-i}{2i-1} \binom{48i}{8i}
	 \left( \frac{8i}{\alpha n} \right)^{8i} 
	 \le \sum_{i=5}^{\varepsilon n} \left( \frac{en}{i} \left( \frac{en}{2i} \right)^2 2^{48} \left( \frac{8i}{\alpha n} \right)^8 \right)^i\\
	\le \sum_{i=5}^{\varepsilon n} \left( 2^{80} \eps \left( \frac{i}{n} \right)^4 \frac{1}{\alpha^8} \right)^i
	\le \sum_{i=5}^{\varepsilon n} \left( \frac{1}{2} \left( \frac{i}{n} \right)^4 \right)^i  = o(1)\, ,
	\end{align*}
	where the third inequality holds by the choice of 
	$\eps$, and where the last estimate holds, 
	since for $i < \sqrt{n}$ we have $\left( \frac{i}{n} \right)^{4} \le n^{-2}$ 
	and since for $i \ge \sqrt{n}$ we have 
	$\left( \frac{1}{2} \right)^i \le n^{-2}$.
	Now, since Maker creates an $\eps n$-expander
	with positive probability, there must also exist
	a deterministic winning strategy for Maker 
	(see e.g.~\cite{BL2000,HKSS_PosGames}).
\end{proof}

Maker will use $G_{n,p}$ to turn this expander into a connected graph and then find many boosters to finish a Hamilton cycle.
A booster for a graph $G$ is any non-edge $e \not\in E(G)$ such that $G+e$ is either Hamiltonian or the length of the longest path in $G+e$ is larger than in $G$.

\begin{lemma}[see e.g.~Lemma~8.5 in~\cite{Bollobas}]
	\label{lem:booster}
	If $G$ is a connected non-Hamiltonian $R$-expander, then the set of boosters for $G$ has size at least $\frac{R^2}{2}$. 
\end{lemma}

Now we have everything at hand to proceed with the proof
of Theorem~\ref{thm:Ham}.

\begin{proof}[Proof of Theorem~\ref{thm:Ham} for $\log n \le b \le \frac{c n }{\log n}$]
	Given $\alpha>0$, 
	let $c'>0$ and $\eps>0$ be given by Lemma~\ref{lem:Expander} 
	on input $\frac{\alpha}{2}$.
	Then let $c=\frac{c'}{2}$, $C=10^4\eps^{-2}$,  
	$\log n \le b \le \frac{c n }{\log n}$, 
	and $p\geq \frac{Cb}{n}$.
	Let $G_{\alpha}$ be any $n$-vertex graph with $n$ large enough and minimum degree $\delta(G_{\alpha}) \ge \alpha n$,
	and reveal $G_2 \sim \gnp$ on 
	$V(G_{\alpha})$. From now on we condition on the properties from
	Lemma~\ref{lem:gnp2} (with $\beta=\eps$). In particular,
	for $G_1=G_{\alpha} \setminus G_2$ we then have 
	$\delta(G_1) \ge \frac{\alpha n }{2}$.
	Moreover, we condition on the following property,
	for which we will show, analogously to 
	\cite[Lemma~2.11]{FGKN_BiasedGames}, 
	that it holds a.a.s.:
	
\begin{enumerate}
\item[(B)] For every non-Hamiltonian connected $\frac{n}{5}$-expander on at least $8n$ and at most $m_0=\frac{16n}{\varepsilon^2}$ edges there exist at least $\frac{n^2p}{100}$ boosters in $G_2$.
\end{enumerate}	

In order to see that this property holds a.a.s.
fix any non-Hamiltonian connected $\frac{n}{5}$-expander
with the mentioned number of edges. By 
Lemma~\ref{lem:booster}, such an expander must have
at least $\frac{n^2}{50}$ boosters.	
The probability that less than $\frac{n^2p}{100}$ of these boosters are edges of $G_2$ is at most $\exp\left(-\frac{n^2p}{400} \right)$,
by an application of Chernoff's inequality (Lemma~\ref{lem:Chernoff1}). Taking a union bound over all relevant
expanders we see that property (B) fails with probability at most	
	\begin{align*}
	\sum_{m=8n}^{m_0} \binom{\binom{n}{2}}{m} \exp\left( -\frac{n^2p}{400}\right) \le \sum_{m=8n}^{m_0} \exp \left(m \log \left(\frac{en^2}{2m}\right) -n \frac{C b}{400} \right)\\
	\le \sum_{m=8n}^{m_0} \exp \left( \frac{16}{\varepsilon^2} n \log n -\frac{C}{400}  n \log n \right)=o(1)
	\end{align*}
where in the last step we use the choice of $C$. (Note that estimating the number of non-Hamiltonian connected $\frac{n}{5}$-expanders with $m$ edges as $\binom{\binom{n}{2}}{m}$ leads to the lower bound of $b = \Omega(\log n)$.)

\smallskip

We will show that Maker can claim a Hamilton cycle on $G_1\cup G_2$. Let us state Maker's strategy.
	
	\medskip
		
	{\bf Strategy description:}
	Maker's strategy consists of two stages. 
	We briefly describe each of these stages here.
	All further details will be given later in the strategy 
	discussion.

	In {\bf Stage I}, which lasts at most 
	$\frac{15n}{\varepsilon^2}$
	rounds, Maker creates a connected 
	$\varepsilon n$-expander. In order to do so,
	she plays alternately on $G_1$ and $G_2$, always
	assuming Breaker's bias to be $2b$ by the 
	trick of fake moves (Lemma~\ref{lem:fake_moves_MB}).
	She only plays on each of these boards until she 
	reaches the following goals.
	On $G_1$ Maker builds an $\varepsilon n$-expander; 
	on $G_2$ Makers occupies a graph in which between any 
	two sets of size $\varepsilon n$ there is at least one 
	edge.
	
	Afterwards, in {\bf Stage II}, 
	Maker uses the remaining free edges of 
	$G_2$ to claim boosters as long as her graph is not 
	Hamiltonian. This stage lasts less than $n$ rounds.
	
	\medskip
	
	{\bf Strategy discussion:}	
	We consider {\bf Stage I} first.
	By Lemma~\ref{lem:Expander} and as 
	$2b \le \frac{c'n}{\log n}$,
	Maker can build an $\varepsilon n$-expander in at most 
	$16 n$ rounds, playing only on $G_1$.
	For her goal on $G_2$ it suffices to 
	ensure that Maker claims an edge in each set from
	\[\mathcal{F} = \{ E_{G_2}(A,B) \colon A,B \subseteq V \text{ disjoint and } |A|=|B|=\varepsilon n  \} .\] 	
	Therefore, we consider the transversal game on $(E(G_2),\mathcal{F}^\ast)$ and, in order to bound the number of rounds
	for winning this auxiliary game, 
	we let Maker play against a bias of 
	$b'=\frac{\varepsilon^2np}{8}$.
	We observe that,
	due to the properties from Lemma~\ref{lem:gnp2}, 
	we have
	$e_{G_2}(A,B) \ge \frac{\varepsilon^2 n^2 p }{2}$ for 
	all 
	disjoint sets $A,B \subseteq V$ with 
	$|A|=|B|=\varepsilon n$. 
	In particular, we obtain
	\[ \sum_{F \in \mathcal{F}} 2^{-|F|/b'+1} 
	\le 2^{2n} 2^{-\varepsilon^2 n^2 p /(2b')+1} =o(1). \]
	From this it follows 
	by the transversal version of Beck's Criterion (Corollary~\ref{cor:beck_transversal}) 	
	that
	Maker wins the $(E(G_2),\mathcal{F}^\ast)$-game with bias $b'$.
	Furthermore, with the trick of fake moves 
	(Lemma~\ref{lem:fake_moves_MB}) and since $e(G_2) \leq n^2 p$ due to conditioning on Lemma~\ref{lem:gnp2}, 
	it follows that Maker can 
	even win this game against a bias of 
	$2b \le \frac{2np}{C} < b'$ in at most 
	$\frac{8n}{\varepsilon^2}$ rounds. 
	Thus, the first stage lasts
	at most $\frac{8n}{\varepsilon^2} + 16n < \frac{15n}{\varepsilon^2}$
	rounds in total.
	
	\medskip	
	
	When Maker enters {\bf Stage II},
	her graph is a connected $\frac{n}{5}$-expander.
	Indeed, if $A\subset V$ is any set 
	with $\eps n \leq |A| \leq \frac{n}{5}$, 
	then, by having a transversal of $\mathcal{F}_2$, 
	there exist less than $\eps n$
	vertices in $V\setminus A$ which 
	are not in the neighbourhood of $A$,
	and hence, $|N_\maker(A)|\geq n - |A| - \eps n > 2|A|$.
	Moreover, at this point and until the end of the 
	second stage (which lasts at most $n$ rounds) 
	at most $\frac{16n}{\varepsilon^2}$ 
	rounds were played. 
	As long as Maker's graph
	is not Hamiltonian, property (B)
	provides at least $\frac{n^2p}{100}$ 
	boosters in $G_2$ and, thus, there are at least
	\[ \frac{n^2p}{100} - m_0 (b+1) 
	\ge \left( \frac{C}{200} - \frac{16}{\varepsilon^2} 
	\right) n(b+1) > 0 \]
	boosters in $G_2$ that are not covered by the graphs of 
	Maker and Breaker.
	Maker picks one booster and repeats this argument
	until she obtains a Hamilton cycle.
\end{proof}

\medskip

\subsection{Waiter-Client Hamiltonicity game}
\label{subsec:Ham_WC}

\begin{proof}
	Similarly to the the discussion of the Maker-Breaker Hamiltonicity game, we consider the two cases $b\leq n^{0.49}$ and $b\geq \ln n$
	separately.
	
	\medskip

{\bf Case A ($b\leq n^{0.49}$):}
Using the same setup as in the proof of 
Theorem~\ref{thm:Ham} for $b\leq n^{0.49}$ (Subsection~\ref{sec:Ham_bsmall}),
consider the edge sets of $G_1$ and $G_2$ as 
two edge disjoint boards. 
Waiter first plays on $G_1$ in such a way
that Client claims a subgraph
with property~(1). 
Afterwards, Waiter plays on $G_2$
in such a way that Client claims
a subgraph with property~(2).
As in Subsection~\ref{sec:Ham_bsmall}, it will be enough to show that Waiter can follow the strategy and reach her respective goals on $G_1$ and $G_2$. Again, we consider the two boards $E(G_1)$ and $E(G_2)$ separately.

\medskip

\textit{$(1:b)$ game on $E(G_1)$:}
Split $G_1$ into the edge-disjoint subgraphs
$$
G_{1,1} = G_1[V_1,V_2] \cup G_1[V_2,V_3] \cup G_1[V_3,V_4]
 \cup G_1[V_4,V_1]
 ~~ \text{ and } ~~
G_{1,2} = G_1[V_1,V_3] \cup G_1[V_2,V_4]
$$
as before. 
First Waiter plays on $G_{1,1}$.
She ensures that
every vertex will have degree at least 
$\frac{\alpha n}{200b}$
in Client's graph. For this, she plays as follows:
as long as there is some $i\in [4]$ and
some vertex $v\in V_i$ with $d_{\client}(v) < \frac{\alpha n}{200b}$,
she offers $(b+1)$ edges between $v$ and $V_{i+1}$
(where we set $V_5=V_1$).
Afterwards, Waiter plays on $G_{1,2}$. Here she ensures 
that Client claims an edge in each edge set contained in
the family $\mathcal{F}_1$ from~\eqref{Ham:F1set}.
By the same calculation as in~\eqref{Ham:F1estimate}
we obtain
$$
\sum_{F\in \mathcal{F}_1} 2^{-|F|/(2b-1) + 1}
<
\sum_{F\in \mathcal{F}_1} 2^{-|F|/(4b) + 1}
= o(1)
$$
and hence, following
Theorem~\ref{thm:WC_transversal},
Waiter can reach her goal.
As before, because of the achievements from the games on $G_{1,1}$ and $G_{1,2}$, Waiter ensures
that Client occupies a graph with property~(1).

\medskip

\textit{$(1:b)$ game on $E(G_2)$:}
Waiter's goal is to make Client occupy 
a spanning subgraph of $G_2$
which fulfils property (2).
For this, we consider the same family 
$\mathcal{F}_2$
as in~\eqref{Ham:F2set}. Using the calculation
from~\eqref{Ham:F2estimate} we have
$$ \sum_{F \in \mathcal{F}_2} 2^{-|F|/(2b-1)+1} = o(1) ,$$
and hence, following Theorem~\ref{thm:WC_transversal},
Waiter can ensure that Client claims an edge 
in each set $F\in \mathcal{F}_2$,
which gives property~(2). 

\medskip\medskip

{\bf Case B ($b \geq \ln n$):} In this case we make use of some ideas from~\cite{BHKL2016}.
	Given $\alpha>0$, let
	$\eps = 10^{-4}\alpha,~ c = \eps^4,~ C = 10^3\eps^{-4}$.
	Further let $G_{\alpha}$ be any graph on $n$ vertices with 
	$n$ large enough 
	and minimum degree $\delta(G_{\alpha}) \ge \alpha n$, let $\ln n\leq b\leq cn$,
	and then with $p \ge \frac{C b}{n}$ 
	reveal $G_2 \sim G_{n,p}$ on $V(G_{\alpha})$.
	From now on we
	condition on the properties from
	Lemma~\ref{lem:gnp2} (with $\beta=\eps$).	
	We set $G_1=G_{\alpha} \setminus G_2$ and, as before,
	we have  $\delta(G_1) \ge \frac{\alpha n}{2}$.
	We will show that Waiter has a strategy to force
	Client to create a Hamilton cycle on $G_1\cup G_2$.
		
	Before we describe Waiter's strategy, let us fix a 
	partition 
	$V = V_1^{(1)}\cup V_1^{(2)} \cup V_2^{(1)}
		\cup V_2^{(2)} \cup V_3^{(1)}\cup V_3^{(2)}$ 
	with 
	$\frac{n}{6}-1\leq |V_i^{(j)}|\leq \frac{n}{6} + 1$,
	for every $i\in [3]$ and $j\in [2]$,
	such that in $G_1$ every vertex has degree at least 
	$\frac{\alpha n}{25}$ into each part $V_i^{(j)}$.
	The existence of such a partition is given by 
	Lemma~\ref{lem:vertex-partition}.
	
	\medskip
	
	{\bf Strategy description:} 
	Waiter's strategy consists of three 
	stages. We briefly describe each of the stages here.
	All further details will be given later in the strategy 
	discussion.

	In {\bf Stage I}, Waiter plays on $G_1$ 
	for at most $\frac{n}{\eps^2}$ rounds. 
	Here, she ensures that
	Client occupies a graph with the following property: 
	
	\begin{enumerate}
	\item[(P)]
	for every $i\in [3]$ and every $A\subset V_i^{(1)}\cup V_i^{(2)}$
	of size at most $\eps n$, 
	there are at least $9 |A|$ neighbours in 
	$V_{i+1}^{(1)} \cup V_{i+1}^{(2)}$  
	(where we set $V_4^{(j)}:=V_1^{(j)}$).
	\end{enumerate}
		
	Afterwards, in {\bf Stage II}, Waiter plays on $G_2$ for at most 
	$\frac{10n}{\eps^2}$ further rounds.
	Now she ensures that Client occupies a graph which has an edge between 
	any two disjoint sets of
	size $\eps n$.
	We will see later that by the end of the second stage Client's graph is a connected 
	$\frac{n}{5}$-expander.
	
	Finally, in {\bf Stage III}, by offering only 
	boosters, Waiter turns this expander
	into a Hamiltonian graph within less than $n$ rounds.
	
	\medskip	
	
	{\bf Strategy discussion:} If Waiter can follow the proposed 
	strategy, it is clear that she forces Client to occupy a
	Hamilton cycle. Hence, it remains to show that she can indeed	do so.
		
	Let us start with the discussion of {\bf Stage I}. 
	Here we assume that Waiter plays
	with bias $b'=\eps^2n > b$, so that we can use the trick of fake moves (Lemma~\ref{lem:fake_moves_WC})
	later on, in order to obtain a good upper bound on 
	the number of rounds for this stage.
	
	By disjointness
	Waiter can play on each of the boards
	$E_{G_1}(V_i^{(1)}\cup V_i^{(2)},V_{i+1}^{(1)}\cup V_{i+1}^{(2)})$
	with $i\in [3]$
	separately, and by symmetry it suffices to give a strategy
	for obtaining~(P) with $i=1$,
	by playing on the board	
	$E_{G_1}(V_1^{(1)}\cup V_1^{(2)},V_{2}^{(1)}\cup V_{2}^{(2)})$.	
	Waiter's strategy is as follows.
	
	First, playing only on $G_1[V_1^{(1)}\cup V_1^{(2)},V_2^{(1)}]$,
	Waiter ensures that Client occupies a transversal of
	\[ \mathcal{F}_1 = \left\{ E_{G_1}(A,B) \colon 
		A\subseteq V_1^{(1)}\cup V_1^{(2)}, 
		B \subseteq V_2^{(1)}, 
		|A|=\eps n \text{, and }|B|= \frac{n}{6} - \frac{\alpha n}{100} \right\} .\]
	Notice that for any $A$ and $B$ as described in $\mathcal{F}_1$	
	we have 
	$$d_{G_1}(v,B)\geq 
	d_{G_1}\left(v,V_2^{(1)}\right) - |V_2^{(1)}\setminus B|
		\geq \frac{\alpha n}{25} - \frac{\alpha n}{100} - 1 
		> \frac{\alpha n}{50}$$
	for all $v\in A$, and hence
	$$ e_{G_1}(A,B) > \frac{\eps\alpha n^2}{50} .$$
	In particular,
	\[ \sum_{F \in \mathcal{F}_1} 2^{-|F|/(2b'-1) + 1} \le 2^{2n} 2^{-\eps \alpha  n^2/(100b') + 1} =o(1), \]
	by the choice of $b'$ and $\eps$.
	Thus, Waiter can force a transversal of 
	$\mathcal{F}_1$ according to Theorem~\ref{thm:WC_transversal}.
	Note that this already gives 
	$|N_\client(A,V_2^{(1)})|\geq \frac{\alpha n}{100} - 1 > 2\eps n = 2|A|$
	for every $A\subseteq V_1^{(1)}\cup V_1^{(2)}$ of size $\eps n$.
	
	Afterwards, as the second step, 
	consider the inclusion minimal sets 
	$A \subseteq V_1^{(1)} \cup V_1^{(2)}$, such that 
	$|A| \le  \eps n$ and 
	$\left| N_{\client}(A, V_2^{(1)}) \right| < 9 |A|$.
	Denote the family of all these sets $A$ by $\mathcal{A}$,
	and let $x_1,\dots,x_r$ be all the vertices contained in these sets.
	Then $r\leq 2\eps n$ by the following reason: 
	If otherwise $r>2\eps n$, then 
	we could find sets
	$A_1,\ldots,A_k\in \mathcal{A}$ for some 
	$k\in\mathbb{N}$
	such that 
	$\eps n \leq \sum_{i\in [k]} |A_i| \leq 2\eps n$,
	and
	hence a set
	$A'\subset \bigcup_{i\in [k]} A_i$ of size $\eps n$.
	Then, by having a transversal of $\mathcal{F}_1$ in Client's 
	graph,
		$|N_\client(A',V_2^{(1)})|  
		\geq \frac{\alpha n}{100} - 1 > 20\eps n$;
	while on the other hand, by the definition of $\mathcal{A}$,
		$|N_\client(A',V_2^{(1)})|\leq 
		\sum_{i\in [k]} |N_\client(A_i,V_2^{(1)})|		
		\leq \sum_{i\in [k]} 9|A_i|
		\leq 18\eps n$,
	a contradiction.

	For every $i=1,\dots,r$ Waiter now plays 9 additional rounds
	where she offers
	exactly $9(b'+1)$ edges between 
	$x_i$ and the set 
	\[Y_i=\{ y \in V_2^{(2)}  \colon 
		d_\client(y,V_1^{(1)} \cup V_1^{(2)})=0, ~ x_iy\in E(G_1) \}\]
	(which is updated after every move), 
	therefore forcing $d_\client(x_i,V_2^{(2)}) \ge 9$.
	This is possible, because at any moment it holds 
	that $|Y_i| \ge d_{G_1}(x_i,V_2^{(2)}) - 9r 
	\geq \frac{\alpha n}{25} - 9r \ge 9(b'+1)$.
	
	Afterwards, the vertices $x_i$ have pairwise disjoint neighbourhoods 
	of size 9 in $V_2^{(2)}$, and it follows immediately
	that after this first stage, property (P) holds for Client's graph (for $i=1$).
	Indeed, let $A \subseteq V_1^{(1)} \cup V_1^{(2)}$ 
	be any set with $1 \leq |A| \leq \eps n$.
	We can partition $A$ into inclusion minimal sets 
	$X_1,\dots,X_s$ such that 
	$|N_\client(X_i,V_2^{(1)})|<9|X_i|$
	and at most one set $B$ with 
	$|N_\client(B,V_2^{(1)})| \ge 9 |B|$.
	Since the vertices $x \in \bigcup_{i=1}^s X_i$ 
	satisfy 
	$d_\client(x_i,V_2^{(2)}) \ge 9$ 
	and the neighbourhoods are disjoint, 
	we get
	\[ N_\client(A,V_2^{(1)} \cup V_2^{(2)})| 
		\ge  |N_\client(B,V_2^{(1)})| 
			+ \sum_{i=1}^s |N_\client(X_i, V_2^{(2)})| 
		\ge 9 |B| + 
	\sum_{i=1}^s 9 |X_i| = 9 |A|~ . \]

	Hence, Waiter can follow the proposed strategy for Stage I
	and force a graph with property~(P).
	Because she succeeds even when the bias 
	equals $b'$, it follows by the trick of fake moves
	(Lemma~\ref{lem:fake_moves_WC}), that Stage I with bias $b$
	can be done in less than $\frac{e(G_1)}{b'} < 
		\frac{n}{\eps^2}$ rounds.
	
	Moreover, we observe that
	property~(P) implies that Client's graph is an 
	$\eps n$-expander.	
	Indeed, consider any $S\subset V$ of size at most $\eps n$.
	Then there is some $i\in [3]$ such that
	$A:=S\cap (V_i^{(1)}\cup V_i^{(2)})$ has size at least 
	$\frac{|S|}{3}$. By property (P) we then conclude that
	\[
	|N_\client(S)| \geq |N_\client(A)| - |S|
		\geq 9|A| - |S| \geq 3|S| - |S| = 2|S|~ .
	\]

	Let us consider {\bf Stage II} now. 
	Here, in order to bound the number of rounds,
	we will assume that
	Waiter plays with bias $b'' = \frac{\eps^2np}{10} > b$.
	Consider the family
	\[\mathcal{F}_2 = 
	\{ E_{G_2}(A,B) \colon A,B \subseteq V \text{ disjoint and } 
		|A|=|B|=\varepsilon n  \} .\] 	
  Waiter again wants Client to occupy a transversal
	of $\mathcal{F}_2$.
	By the properties from Lemma~\ref{lem:gnp2}
	we have $|F|\geq 0.5 p (\eps n)^2$
	for every $F\in \mathcal{F}_2$, and therefore
	\[ \sum_{F \in \mathcal{F}_2} 2^{-|F|/(2b''-1) + 1} \le 
	2^{2n} 2^{- \eps^2 n^2 p/(4b'') + 1} =o(1)\]
	by the choice of $b''$. Thus, Waiter can follow this part of the proposed strategy.
	Note that this stage lasts at most $\frac{10n}{\eps^2}$
	rounds.	Indeed, since Waiter succeeds even when the 
	bias 
	equals $b''$, it follows by the trick of fake moves
	(Lemma~\ref{lem:fake_moves_WC}), that Stage II with 
	bias $b$
	can be done in less than $\frac{e(G_2)}{b''} 
		\leq \frac{10n}{\eps^2}$ 
	rounds,
	where we use that $e(G_2) \leq n^2p$ by the properties 
	from Lemma~\ref{lem:gnp2}.

	To finish the discussion of Stage~II notice 
	that Client's graph
	is a connected $\frac{n}{5}$-expander 
	analogously to the discussion of
	the Maker-Breaker Hamiltonicity game
	for $\ln n \leq b \leq \frac{cn}{\ln n}$.
		
	\smallskip		
		
	Finally, for {\bf Stage~III}, similarly 
	as in the proof of Theorem~\ref{thm:Ham}
	(for $\ln n \leq b \leq \frac{cn}{\ln n}$), we conclude that as 
	long as 
	Client's graph $\client$ is not Hamiltonian, 
	there are at least $b+1$ boosters available.
	Waiter offers these boosters to Client and, 
	after at most $n$ additional rounds, 
	there exists a Hamilton cycle in Client's 
	graph. 
\end{proof}

\bigskip


\section{$k$-Connectivity game} 
\label{sec:conn}

\subsection{Maker-Breaker $k$-vertex-connectivity game}

In this subsection we will prove Theorem~\ref{thm:vertex-connectivity}.
In order to prove the theorem, we will make use of the following two results, the first providing a partition of $G_{\alpha}$ into highly connected components, the second providing a necessary condition for a graph to be $r$-edge-connected.

\begin{lemma}[Lemma~1 in~\cite{BFKM2004}]\label{lem:partition}
Let $G$ be any graph on $n$ vertices
and $\delta(G)\geq \alpha n$. Then there exists an integer $s$ and a partition $V(G)=V_1\cup V_2\cup \ldots \cup V_s$
such that the following holds:
\begin{enumerate}
\item[(a)] $|V_i|\geq \frac{\alpha}{8}n$ for every $i\in [s]$, and
\item[(b)] $G[V_i]$ is $\frac{\alpha^2}{16}n$-vertex-connected
for every $i\in [s]$.
\end{enumerate}
\end{lemma}

\begin{thm}[Theorem~6.2 in~\cite{Karger1993}]\label{thm:karger}
There exists a constant $d>0$ such that the following holds: Let $G$ be an $r$-edge-connected graph on $n$ vertices. 
Then for every $t \geq 1$ there are at most $dn^{2t}$ cuts of size less than $rt$ in $G$.
\end{thm}

\begin{proof}[Proof of Theorem~\ref{thm:vertex-connectivity}]

Let $\alpha>0$ and $k\geq 1$ be given. Choose
$C = 10^5 k^3 \alpha^{-3} \ln\left(8k\alpha^{-1}\right) $ and
$c = \alpha^4 10^{-5} k^{-2}$, and let
$1 \le b \le \frac{cn}{\ln n}$ and $p \ge \frac{Cb}{n^2}$.
Before we describe a strategy for Maker, let us observe that
$G \sim G_{\alpha}\cup G_{n,p}$ satisfies the following properties with
probability at least $1-\exp(-cpn^2)$:
\begin{itemize}
\item[(i)] there exists an integer $s$ and a partition $V(G)=V_1\cup V_2\cup \ldots \cup V_s$ as stated in Lemma~\ref{lem:partition},
\item[(ii)] for each $j\in [s]$ there exists a partition
$V_j = V_{j,1}\cup  V_{j,2}\cup \ldots \cup V_{j,k}$, such that
$e_G(V_{i,\ell},V_{s,\ell}) \geq \frac{\alpha^2 n^2 p}{200k^2} $ for every $i\in [s-1]$ and $\ell \in [k]$.
\end{itemize}

Indeed, Lemma~\ref{lem:partition} applied to $G_{\alpha}$ gives a partition $V(G_{\alpha})=V_1\cup V_2\cup \ldots \cup V_s$
as desired.
For any $j\in [s]$, fix any partition
$V_j = V_{j,1}\cup  V_{j,2}\cup \ldots \cup V_{j,k}$
such that 
$|V_{j,\ell}|\geq \lfloor \frac{|V_j|}{k} \rfloor
	\geq \frac{\alpha n}{10k}$.
Only afterwards reveal the edges of $G'\sim G_{n,p}$
on $V(G_{\alpha})$, and set $G=G_{\alpha}\cup G'$.
Adding the edges of $G'$ to $G_{\alpha}$ does not destroy the properties~(a) and~(b) from Lemma~\ref{lem:partition},
and hence
$V(G)=V_1\cup V_2\cup \ldots \cup V_s$ stays a partition as required for~(i).
Moreover, notice that (a) implies $s\leq \frac{8}{\alpha}$.
In order to get (ii), observe that $e_{G'}(V_{i,\ell},V_{s,\ell})\sim \operatorname{Bin}(|V_{i,\ell}||V_{s,\ell}|,p)$ with 
$\Exp[e_{G'}(V_{i,\ell},V_{s,\ell})]=|V_{i,\ell}||V_{s,\ell}| p \geq \frac{\alpha^2n^2p}{100k^2}$.
Therefore, using Chernoff's inequality (Lemma~\ref{lem:Chernoff1}), we obtain
\begin{align*}
\Prob\left( \exists i\in [s-1], \ell\in[k]:~ e_{G'}(V_{i,\ell},V_{s,\ell}) < \frac{\alpha^2}{200k^2}n^2p \right) 
	& \leq sk \cdot e^{-\alpha^2 n^2p / (800k^2)} \\
	& \leq e^{\ln\left( 8k / \alpha \right)-(10^{-4}k^{-2}\alpha^2 + c)n^2p} 
		\leq e^{-cn^2p} ~ ,	
\end{align*}
where the second inequality follows from the choice of $c$ and since $s\leq \frac{8}{\alpha}$, and
the last inequality follows since $p\geq \frac{Cb}{n^2} \ge \frac{C}{n^2}$ and by the choice of $C$.  

\medskip

From now on we will condition on the properties~(i) and~(ii).
Next, we will first describe a strategy for Maker in a $(1:b)$ game on $G$ and then we will show that she can follow that strategy
and that it leads to a $k$-vertex-connected spanning subgraph of $G$. 

\medskip

{\bf Strategy description:} In order to describe the strategy, consider the $2s-1$ edge-disjoint boards
$E_{G}(V_i)$ for all $i\in [s]$, and $E_G(V_i,V_s)$ for all $i\in [s-1]$.
Enumerate all these boards in an arbitrary way
with the integers $0,1,\ldots,2s-2$. In round $i$, let Maker play on board $i$ (mod~$2s-1$). Then, between any two moves on the same board, Breaker claims at most $(2s-1)b =:b_s$ 
edges. Hence, using the trick of fake moves (see Lemma~\ref{lem:fake_moves_MB}), we can assume that on each board we separately play a $(1:b_s)$ Maker-Breaker game.
Now, on each of the boards $E_G(V_i)$ Maker plays in such a way that she obtains
a $k$-vertex-connected spanning subgraph of $G[V_i]$,
and on each of the boards $E_G(V_i,V_s)$ she makes sure to claim
a matching of size at least $k$. All details will be given later in the strategy discussion.

\medskip

{\bf Strategy discussion:} Notice that, if Maker can follow the strategy,
then her final graph $\maker$ will be $k$-vertex-connected.
Indeed, let $K\subset V(\maker)$ be any subset of size at most $k-1$.
Then $\maker[V_i\setminus K]$ is connected, since $\maker[V_i]$ is $k$-vertex-connected, for every $i\in [s]$, and due to the matchings of size $k$, there is at least one edge between $V_s$ and $V_i$ in the graph $\maker-K$, for every $i\in [s-1]$. So, $\maker-K$ is connected.

\smallskip

Now, it remains to show that Maker can indeed follow her strategy. We consider each of the above mentioned
boards separately. 

\medskip

{\em $(1:b_s)$ game on $E_G(V_i)$:} Maker's goal is to
occupy a spanning $k$-vertex-connected subgraph of $G[V_i]$. This part of her strategy is motivated by~\cite[Theorem~1.6]{AHK2010}. 
Let us define
$$
\mathcal{F} := \big\{ E_{G}(S,V_i\setminus (S\cup K)):~ 
	K\subset V_i,~ 0\leq |K|\leq k-1,~ S\subset V_i\setminus K
	\big\}.
$$
If Maker manages to occupy
a transversal of $\mathcal{F}$, then
the following holds for her final subgraph $\maker_i$
of $G[V_i]$: for every subset $K\subset V_i$
of size at most $k-1$, the graph $\maker_i-K$
is connected since there is at least one edge in
every cut $(S,V_i\setminus (S\cup K))$ of $\maker_i-K$.
That is, $\maker_i$ is $k$-vertex-connected.

Now, according to Corollary~\ref{cor:beck_transversal},
it suffices to show that
$$
\sum_{F\in \mathcal{F}} 2^{-|F|/b_s + 1} < 1~ .
$$

For this notice that for any $K\subset V_i$ of size at most $k-1$
the following holds: $G[V_i]$ is $\frac{\alpha^2}{16}n$-vertex-connected
and hence the vertex-connectivity of $G[V_i\setminus K]$
is at least $\frac{\alpha^2}{16}n - (k-1) > \frac{\alpha^2}{20}n$.
In particular,
$e_G(S,V_i\setminus (S\cup K))\geq \frac{\alpha^2}{20}n$
for every $S\subset V_i\setminus K$.
Now, applying Theorem~\ref{thm:karger} (with $r=\frac{\alpha^2}{20}n$ and $t=\lceil \frac{j+1}{r} \rceil$), we obtain
\begin{align*}
\sum_{F\in \mathcal{F}} 2^{-|F| / b_s + 1}
	& \leq 
	\sum_{K\subset V_i,\atop |K|\leq k-1}	
	\sum_{j= \frac{\alpha^2}{20}n}^{n^2} |\{S\subset V_i\setminus K:~ 
	e_G(S,V_i\setminus (S\cup K))=j\}|
		\cdot 2^{- j / b_s + 1}  \\
	& \leq 
	\sum_{K\subset V_i,\atop |K|\leq k-1}
	\sum_{j= \frac{\alpha^2}{20}n}^{n^2} 
		d|V_i\setminus K|^{2 \lceil (j+1) / r \rceil} 
			\cdot 2^{- j / (2b_s)} \\
	& \leq n^{k} \sum_{j= \frac{\alpha^2}{20}n}^{n^2} 
		n^{4j / r} \cdot  2^{-j / 2b_s} \\
	& \leq n^{k} \sum_{j= \frac{\alpha^2}{20}n}^{n^2}
		\exp\left( \frac{4}{r}\cdot \ln n -\frac{\ln 2}{2b_s} \right)^j 		\\
	& \leq 
		n^{k} \sum_{j= \frac{\alpha^2}{20}n}^{n^2}
		\exp\left(  -\frac{100k}{\alpha^2} \cdot \frac{\ln n}{n} \right)^j \\
	& \leq n^{k} \sum_{j= \frac{\alpha^2}{16}n}^{n^2}
		\exp\left(  - 5k \ln n  \right) = o(1)
\end{align*}
where the fifth inequality holds by the choice of $c$ and since $b_s<2s b \leq \frac{16c}{\alpha} \cdot \frac{n}{\ln n}$.

\medskip

{\em $(1:b_s)$ game on $E_G(V_i,V_s)$:} In order for Maker to claim 
a matching of size $k$ between $V_i$ and $V_s$, it is enough to claim one edge between $V_{i,\ell}$ and $V_{s,\ell}$ for every $\ell\in [k]$.
As this takes $k$ rounds in total, at most $k(b_s+1)$ edges
can be claimed in the meantime. Thus, Maker succeeds easily
if $e_G(V_{i,\ell},V_{s,\ell})>k(b_s+1)$
for every $\ell\in[k]$. The latter is the case since by (ii) we obtain 
\[
e_G(V_{i,\ell},V_{s,\ell}) 
\geq \frac{\alpha^2}{200k^2}n^2p \geq \frac{\alpha^2 Cb}{200k^2}  
\geq 4skb > k (b_s+1)\, ,
\]
where the third and the last inequality follow by the definition of $C$ and $b_s$, respectively.
\end{proof}

\medskip

\subsection{Waiter-Client $k$-vertex-connectivity game}

In this subsection we will prove Theorem~\ref{thm:vertex-connectivity_WC}.

\begin{proof}[Proof of Theorem~\ref{thm:vertex-connectivity_WC}]
Let $\alpha>0$ and $k\geq 1$ be given. 
We set $\beta = \frac{\alpha^2}{80k}$ and let $\gamma$ be returned by Lemma~\ref{lem:partition_connected}.
Now we choose $\eps = 10^{-4}\gamma$,
$c = 10^{-5}\alpha^4k^{-2}\eps^{3}$ and
$C = 10^6k^3\alpha^{-3}\ln(8k\alpha^{-1})$.

Before we describe Waiter's strategy
we split the board into suitable subboards. 
As a first step, fix a partition
$V(G_\alpha)=U_1\cup U_2\cup \ldots \cup U_k$  
such that
	\begin{enumerate}
	\item[(1)] $\frac{n}{k}-1 \leq |U_i|
	   \leq \frac{n}{k} + 1$ for all $i \in [k]$,
	\item[(2)] $e_{G_{\alpha}}(v, U_i) \geq \frac{\alpha}{2}|U_i|$ for all $v \in V(G)$ and $i \in [k].$
		\end{enumerate}
Such a partition exists by Lemma~\ref{lem:vertex-partition}.
Next, we additionally split each of the sets $U_i$, $i\in [k]$, to obtain a partition
$U_i=U_{i,1}\cup U_{i,2} \cup \ldots \cup U_{i,s_i}$
such that
\begin{enumerate}
\item[(a)] $|U_{i,j}|\geq \frac{\alpha}{20k}n$, and
\item[(b)] $G_{\alpha}[U_{i,j}]$ is $\beta n$-vertex-connected
\end{enumerate}
for every $j\in [s_i]$. Such a partition can be found by Lemma~\ref{lem:partition}.
For every $i\in [k]$ and $j\in [s_i]$, set $G_{i,j}:=G_{\alpha}[U_{i,j}]$. Applying Lemma~\ref{lem:partition_connected} we can find a partition
$G_{i,j}=G_{i,j}^1\cup G_{i,j}^2$ such that
both parts are 
$\gamma n$-vertex-connected graphs on $U_{i,j}$.
Only afterwards, we reveal the edges of $G'\sim G_{n,p}$ and observe that
with probability at least $1-\exp(-cn^2p)$ the following holds:
\begin{enumerate}
\item[(c)] $e_{G'}(U_{i,j_1},U_{i,j_2}) \geq 
\frac{\alpha^2}{800k^2} n^2 p$ for every $i$ and $j_1\neq j_2$.
\end{enumerate}
The proof of (c) is analogous to the discussion of (ii) in the proof of Theorem~\ref{thm:vertex-connectivity}.
From now on, we will condition on (c) being satisfied,
and show that Waiter has a strategy to
force a $k$-vertex-connected spanning subgraph of $G=G_{\alpha}\cup G'$.

\medskip

{\bf Strategy description:} Waiter's strategy consists of four stages. We briefly describe each of these stages here
by mentioning the board and the goal of the stage.
All further details will be given later in the strategy discussion.

In {\bf Stage~I}, Waiter plays on the board
$G_{I}:=\bigcup_{i,j} E(G_{i,j}^1)$. Here she ensures that Client
creates an $\eps n$-expander on $V(G)$. 
If Waiter succeeds, then immediately afterwards each component in Client's graph has size at least $\eps n$ and hence there are at most $\eps^{-1}$ such components. Furthermore, each of these components 
must be a subset of some set $U_{i,j}$ with $i\in [k]$ and $j\in [s_i]$.
In {\bf Stage~II}, Waiter plays on the board
$G_{II}:=\bigcup_{i,j} E(G_{i,j}^2)$ where she makes sure that  Client's graph becomes connected on each of the sets $U_{i,j}$. 
Next, in {\bf Stage~III}, Waiter plays on the board
$G_{III}:=\bigcup_{i,j_1,j_2} E_{G'}(U_{i,j_1},U_{i,j_2})$.
She forces Client to make $U_i$ a connected component in her graph for every $i\in [s]$. 
In {\bf Stage~IV}, Waiter considers the board
$G_{IV}:=\bigcup_{i_1,i_2} E_{G_{\alpha}}(U_{i_1},U_{i_2})$.
She ensure, that by the end of this stage 
Client's graph $\client$ satisfies the following:
$e_{\client}(v,U_{i_2})>0$ for every 
$i_1\neq i_2$ and $v\in U_{i_1}$.

\medskip

{\bf Strategy discussion:}
If Waiter can follow the proposed strategy, 
Client's graph will be $k$-vertex-connected by the end of the game.
This can be seen as follows. Let $K\subset V(G)$ be any set of size at most $k-1$. Then there exists some $i\in [k]$ such that
$U_i\cap K = \varnothing$. By Stage~III, $U_i$ is a connected component in Client's graph;
lastly by Stage~IV, every other vertex in $V\setminus (K\cup U_i)$ has a neighbour in $U_i$; i.e.~$\client-K$ is connected.

\medskip

Hence, it remains to be shown that Waiter can follow the proposed strategy. We will discuss each stage separately.
However, before doing so, observe that the four different boards $G_{I}\ldots,G_{IV}$ are pairwise disjoint,
i.e.~Waiter can play on these boards one after the other.

\medskip

{\bf Stage~I.} Since $G_{i,j}^1$ is $\gamma n$-vertex-connected
for every $i,j$, we have $\delta(G_{I})\geq \gamma n$. 
Waiter follows the strategy from Stage I in Case B of the proof of Theorem~\ref{thm:Ham_WC} and thus forces an $\eps n$-expander.
Note that the mentioned strategy of Case B only used the fact that
the game was played on a graph with minimum degree at least $\frac{\alpha n}{2}$ (which is replaced here with $\gamma n$);
the strategy worked for a bias $b'=\eps^2 n$ 
(now $\eps$ is chosen depending on $\gamma$ instead of $\alpha$) and the same strategy can be used for any smaller bias by the trick of fake moves (Lemma~\ref{lem:fake_moves_WC}).

\medskip

{\bf Stage~II.} When Waiter enters Stage II, Client's graph consists of at most $\eps^{-1}$ components, each of which has size at least $\eps n$ and is contained in one of the sets $U_{i,j}$. Let $i\in [k]$ and $j\in [s_i]$ be fixed.
Waiter plays on $E(G_{i,j}^2)$ as follows.
As long as Client's graph on $U_{i,j}$ is not connected,
Waiter looks for two components $A,B\subset U_{i,j}$
such that there are at least $b+1$ free edges in 
$G_{i,j}^2[A,B]$. Waiter then offers these edges to Client and thus reduces the number of components by 1; she repeats this step until $U_{i,j}$ is a connected component in Client's graph.

\medskip

In order to show that Waiter can indeed follow this strategy, it remains to be shown that we can always find such sets $A,B$ as described above. We do this as follows.
Since Client's graph is assumed to be disconnected on $U_{i,j}$ there must be two vertices $v,w\in U_{i,j}$
which belong to different components.
Now $G_{i,j}^2$ is $\gamma n$-vertex-connected,
and hence in this graph there must be $\gamma n$ internally vertex-disjoint paths between $u$ and $v$.
Each of these paths must have at least one edge connecting
two of Client's components. Since there are at most $\eps^{-1}$ components, there must be 
some pair $A,B$ of components, between which $G_{i,j}^2$ has
at least $\eps^{2} \gamma n \geq b+1$ edges.
Waiter can offer $b+1$ of these edges and thus follow her strategy. Note that this way she does not offer 
edges between any other pair of components,
which makes it possible to repeat this argument until
$U_{i,j}$ is a connected component.

\medskip

{\bf Stage~III.}
When Waiter enters Stage III, Client's components are the sets $U_{i,j}$. Now, for each $i\in [k]$ and distinct 
$j_1,j_2\in [s_i]$, Waiter plays exactly one
round on $E_{G'}(U_{i,j_1},U_{i,j_2})$ 
offering $b+1$ arbitrary edges. This is possible by (c)
and since 
$$
\frac{\alpha^2n^2p}{800k^2} 
\geq 
\frac{\alpha^2Cb}{800k^2} > b+1
$$
by the choice of $C$. One can easily check that
this makes Client's graph connected
on each set $U_i$ with $i\in [k]$.

\medskip

{\bf Stage~IV.}
Let distinct $i_1,i_2\in [k]$ be fixed.
Waiter plays on the graph 
$G_{\alpha}[U_{i_1},U_{i_2}]$ 
to make sure that every vertex $v\in U_{i_1}$
gets a neighbour in $U_{i_2}$ in Client's graph,
and vice versa. For this, note that by (1) and (2)
$G_{\alpha}[U_{i_1},U_{i_2}]$  has minimum
degree at least $\frac{\alpha n}{3k}$.
We can split this graph into two subgraphs $H_1$ and $H_2$ 
on the same vertex set $U_{i_1}\cup U_{i_2}$, each having minimum degree at least $\frac{\alpha n}{10k}$. Indeed, a random partition of
the edges of $G_{\alpha}[U_{i_1},U_{i_2}]$ can be used to prove the existence of $H_1$ and $H_2$.
Using only the edges of $H_1$,
Waiter ensures that every vertex in $U_{i_1}$
gets a neighbour in $U_{i_2}$. This is possible
as $d_{H_1}(v,U_{i_2})\geq \frac{\alpha n}{10k}  \geq b+1$
for every $v\in U_{i_1}$. Then, Waiter repeats the same
with $H_2$ and every vertex $v\in U_{i_2}$.
This finishes the proof.
\end{proof}

\bigskip


\section{Unbiased $H$-games}
\label{sec:hgame}

For the study of the $H$-game we will use
regularity tools.
For a first example 
we will prove the following proposition.

\begin{prop}
\label{prop:oddcycle}
Let $H$ be an odd cycle of length at least 5, $\alpha>0$, $G_{\alpha}$ be a graph on $n$ vertices with $\delta(G_{\alpha})\geq \alpha n$, and $p = \omega(n^{-2})$.
Then a.a.s.~the following holds: playing an unbiased Maker-Breaker game on $G_{\alpha}\cup G_{n,p}$, Maker has a strategy to claim a copy of $H$.
\end{prop}

\begin{proof}
Let $H=C_{2\ell +1}$ with $\ell\geq 2$.
An application of Lemma~\ref{lem:reg-turan} with 
$\eps=\frac{\alpha^2}{10^3}$ (and $\delta=\alpha$)
leads to a constant $\eta>0$. 
Now, by Lemma~\ref{lem:reg-turan} we can find an
$\eps$-regular pair $(V_1,V_2)$ in $G_{\alpha}$, with density at least $\frac{\alpha}{2}$ and such that $|V_1|\geq |V_2| \geq \eta n$. 
Let $G=G_{\alpha}[V_1,V_2]$ be the bipartite subgraph of $G_{\alpha}$ with vertex classes $V_1$ and $V_2$.
Let $V_i':=\left\{v\in V_i:~ d_G(v,V_{3-i}) \ge \frac{\alpha}{3}|V_{3-i}| \right\}$. Then, by the $\eps$-regularity of $(V_1,V_2)_G$, we get 
$|V_i'|\ge (1-\eps) |V_i|$ for $i\in [2]$.
Then for any $v\in V_i'$ it holds that
\[
d_G(v,V_{3-i}')
\geq d_G(v,V_{3-i}) - \eps |V_{3-i}|  
\geq \frac{\alpha}{4}|V_{3-i}|~ .
\]

Next, we reveal the edges of $G'\sim G_{n,p}$ on $V(G_{\alpha})$. 
Then a.a.s. there exists an edge, say $xy$, in $V_1'$. 
Indeed, by Chernoff (Lemma~\ref{lem:Chernoff1}) and by the size of $V_1'$, we obtain that
\[
\Prob\left( e_{G'}(V_1') = 0 \right)
\leq \exp \left( -\frac{1}{3} \Exp\left( e_{G'}(V_1') \right) \right)
\leq e^{-\omega(1)}~ .
\]
From now on, we will condition on the existence of such an edge $xy$. We will describe a strategy for Maker, when playing on $G_{\alpha}\cup G'$, and show that she can follow that strategy until a copy of $H$ is created.

\medskip

{\bf Strategy description:}
Maker's strategy consists of two stages.

{\bf Stage~I} lasts exactly
$\frac{\alpha}{16}|V_2|+1$ rounds.
In the first round, Maker claims the edge $xy$. In the proceeding
$\frac{\alpha}{16}|V_2|$ rounds, Maker claims arbitrary edges from $E_G(y,V_2')$.

At the end of this stage, let $v_1=x$, $N_y=N_\maker(y,V_2')$ 
and set
$
V_1^{\ast}=\left\{v\in V_1':~ d_G(v,N_y)\geq \frac{\alpha}{8}|N_y| \right\}~ .
$

Afterwards, {\bf Stage II} lasts exactly $2\ell-1$ rounds. 
For any $k\leq 2\ell-1$, Maker does her $k^{\text{th}}$ move in Stage II as follows:
\begin{itemize}
\item If $v_k\in V_1$ then she fixes an arbitrary vertex 
\[
z\in
\begin{cases}
V_2' & ~ \text{if }k=1,\\
N_y& ~ \text{if }k\neq 1
\end{cases}
\]
such that $d_\breaker(z)\leq \sqrt{n}$ and $v_kz\in G\setminus \breaker$. 
\item If $v_k\in V_2$ then she fixes an arbitrary vertex $z\in V_1^{\ast}$
such that $d_\breaker(z)\leq \sqrt{n}$ and $v_kz\in G\setminus \breaker$.
\end{itemize}
Maker then claims the edge $v_kz$
and sets $v_{k+1}=z$.

\medskip

{\bf Strategy discussion:}
If Maker can follow her strategy, then 
$(v_1,v_2,\ldots,v_{2\ell + 1})$, with 
$v_{2\ell +1}=y$, forms a cycle of length $2\ell+1$ in her graph.
Hence, it remains to be shown that Maker can follow the proposed strategy.

Consider {\bf Stage I} first. Claiming $xy$ in round 1 is not a problem since Maker starts the game. Since $y\in V_1'$, we have 
$d_G(y,V_2')\geq \frac{\alpha}{4}|V_2|.$
Hence, Maker can easily claim 
$\frac{\alpha}{16}|V_2|$ edges from
$E_G(y,V_2')$
in the beginning of the game.

At the end of Stage I, observe that,
according to the slicing lemma (Lemma~\ref{lem:slicing}), 
the pair 
$(V_1',N_y)$ is $\eps'$-regular in $G$ with 
$\eps'=\frac{16\eps}{\alpha}$ and density at least 
$\frac{\alpha}{4}$.
In particular, less than $\eps' |V_1'|$ vertices from $V_1'$ have less than
$\frac{\alpha}{8}|N_y|$ neighbours in $N_y$.
Hence, $$|V_1^{\ast}|\geq (1-\eps')|V_1'|> (1-2\eps')|V_1|~ .$$

Next, let us look at {\bf Stage II}. 
Throughout this stage,
we have the bound 
$$e(\breaker)\leq 1 + \frac{\alpha}{16}|V_2| + (2\ell -1) < \frac{\alpha}{8}|V_i|$$ for $i\in [2]$.
Assume that Maker can follow her strategy until she reaches round $k$ of Stage~II. We explain now why she can follow 
the strategy for the $k^{\text{th}}$ move in Stage~II.

\medskip

If $k=1$, then $v_k=x\in V_1'$ and hence
$$
d_{G\setminus \breaker}(v_k,V_2')
\geq
d_{G}(v_k,V_2') - e(\breaker)
\geq
\frac{\alpha}{3}|V_1| - \frac{\alpha}{8}|V_i|
>
\frac{\alpha}{5}|V_i|
> e(\breaker)~ .
$$
Therefore, there exists a vertex $z$ as required for the strategy, even with $d_\breaker(z)=0$.

If $k\neq 1$ is odd, then by the strategy from the previous round, we know that $v_k\in V_1^{\ast}$ and $d_\breaker(v_k)\leq \sqrt{n}+1$. In this case we have
$$
d_{G\setminus \breaker}(v_k,N_y)
\geq
d_{G}(v_k,N_y) - d_\breaker(v_k)
\geq
\frac{\alpha}{8}|N_y| - (\sqrt{n}+1)
>
\frac{\alpha}{10}|N_y|
$$
for large enough $n$. Hence, there are more than
$\frac{\alpha}{10}|N_y|$ choices for the desired vertex $z$
if we ignore the constraint $d_\breaker(z)\leq \sqrt{n}$. However,
there must be a vertex fulfilling this constraint since
otherwise we have
$$
e(\breaker)>\frac{\alpha}{10}|N_y|\cdot \sqrt{n} \geq \frac{\alpha^2\eta}{160} n^{\frac{3}{2}} > \frac{\alpha}{8}|V_2|
$$
for large enough $n$, in contradiction to the upper bound on $e(\breaker)$ described earlier. 

If otherwise $k$ is even,
then by the strategy from the previous round, we know that $v_k\in N_y\subset V_2'$ and $d_\breaker(v_k)\leq \sqrt{n}+1$. In this case we have
\begin{align*}
d_{G\setminus \breaker}(v_k,V_1^{\ast})
& \geq
d_{G\setminus \breaker}(v_k,V_1) - |V_1\setminus V_1^{\ast}|
\geq
d_{G}(v_k,V_1) - d_\breaker(v_k) - |V_1\setminus V_1^{\ast}| \\
& >
\frac{\alpha}{3}|V_1| - (\sqrt{n}+1) - 2\eps'|V_1|
>
\frac{\alpha}{5}|V_1|
\end{align*}
for large enough $n$. Analogously to the previous case, 
using the upper bound on $e(\breaker)$, 
we can conclude that there must exist a vertex
$z\in N_{G\setminus \breaker}(v_k,V_1^{\ast})$
such that $d_\breaker(z)\leq \sqrt{n}$.
Hence, in any case, Maker can follow the proposed strategy.
\end{proof}

The previous result covers all odd cycles
except for $C_3$. For the sake of completeness, we will
discuss $C_3$ with the following proposition.
A \emph{book with $t$ pages} consists of $t$ triangles overlapping in a single edge.

\begin{prop}
\label{prop:book}
Let $H$ be a book, $\alpha>0$, $G_{\alpha}$ be a graph on $n$ vertices with $\delta(G_{\alpha})\geq \alpha n$, and $p = \omega(n^{-2})$.
Then a.a.s.~the following holds: playing an unbiased Maker-Breaker game on $G_{\alpha}\cup G_{n,p}$, Maker has a strategy to claim a copy of $H$.
\end{prop}

\begin{proof}
Let $H$ be a book with $t$ pages, $t\in\mathbb{N}$. 
For our strategy consider the following graph $F$:
take 2 vertex-disjoint matchings $M_1,M_2$ of size $k=12t$;
set $V(F)= V(M_1)\cup V(M_2)$ and
$$
E(F)=E(M_1) \cup E(M_2) \cup \left\{ vw:~ v\in V(M_1)~ \text{and} ~ w\in V(M_2)\right\}~ . 
$$ 
Then $m^{(2)}(F)=\frac{1}{2}$. Hence, using Theorem 2.1 from~\cite{KST_SmoothedAnalysis} we know that for $p=\omega(n^{-2})$
a.a.s. $G\sim G_{\alpha}\cup G_{n,p}$ contains a copy of $F$.
In the following we will show that, playing only on $F$, Maker has a strategy to occupy a copy of $H$.
At first Maker claims $\frac{k}{3}$ edges of each of the matchings $M_1$ and $M_2$. As Breaker in the meantime can claim at most $\frac{2k}{3}$ edges, Maker can easily do so.
Denote with $M_1'=\left\{e_1,\ldots,e_{\frac{k}{3}}\right\}$ and $M_2'=\left\{f_1,\ldots,f_{\frac{k}{3}}\right\}$ the submatchings of $M_1$ and $M_2$ claimed by Maker.
Afterwards, consider the $\frac{k^2}{9}$ edge-disjoint boards
$E_{i,j}=\left\{ vw:~ v\in V(e_i)~ \text{and} ~ w\in V(f_j)\right\}$ with $i,j\in \left[ \frac{k}{3} \right]$.
Since so far only $\frac{2k}{3}$ rounds have been played,
at least $\frac{k^2}{9} - \frac{2k}{3} \geq \frac{k^2}{18}$
of these boards are free of Breaker's edges.
On each of these boards, Maker now ensures to claim
two adjacent edges by a simple pairing strategy.
This way, Maker creates at least $\frac{k^2}{18}$ triangles,
each of which contains one of the edges from $M_1'\cup M_2'$.
By a simple averaging argument we conclude that at least one of these edges needs to be contained in at least
$\frac{k}{12}=t$ Maker's triangles, hence leading to a copy of $H$.
\end{proof}

Both propositions are optimal in terms of $p$.
However, when $\alpha>\frac12$, then playing on $G_\alpha$ is sufficient and we can set $p=0$, c.f.~Theorem~\ref{thm:hgame0}.
For the proof of this theorem we first prove a useful lemma that allows Maker to maintain regularity in her graph.
We will also require this in the proof of Theorem~\ref{thm:H-game_Maker} and, in fact, Theorem~\ref{thm:hgame0} will easily follow from the proof of Theorem~\ref{thm:H-game_Maker}.

\subsection{Maintaining regular pairs}
\label{subsec:maintain-regularity}

One of the ingredients for Maker's strategy 
in the $(1:1)$ Maker-Breaker $H$-game on $G_{\alpha}\cup G_{n,p}$ is the following lemma.
It roughly states, that if a game is played on the edge set of some graph $G$, in which a given pair $(A,B)$
of disjoint subsets of vertices is regular, then 
Maker can ensure to claim a subgraph for which the pair $(A,B)$ is still regular.

\begin{lemma}
\label{lem:regularity_Maker}
For every reals $0< \eps < \alpha <1$
with $\alpha > 8\eps$
the following holds provided $n$ is large enough.
Let $G=(A\cup B,E)$ be a bipartite graph with $|A|=|B|=n$
such that $(A,B)_G$ is
$\eps$-regular and has density $\alpha$.
Then Maker has a strategy to ensure
that in the $(1:1)$ Maker-Breaker game
on $E(G)$ she occupies a subgraph $\maker\subset G$
such that $(A,B)_\maker$ is $4\eps$-regular and has density $\frac{\alpha}{2}$.
\end{lemma}

The above statement will follow easily from the 
following more general statement on {\em Discrepancy games} due to Hefetz, Krivelevich, and Szab\'o~\cite{HKS2007}. The general setup of such a game is as follows. Let a bias $b$, a constant $\eps>0$ and some hypergraph $(X,\mathcal{H})$ be given. The $(b,1,\mathcal{H})$ $\eps$-Discrepancy game is played by two players,
called Balancer and Unbalancer, who alternately
claim previously unclaimed elements of $X$. Balancer, starting the game, claims $b$ elements of $X$ in every round (except for maybe the last round when there are less then $b$ elements left), while Unbalancer always claims $1$ such element. Denote with $B$ the set of all elements claimed by Balancer by the end of the game.
Then Balancer is called the winner if and only if
\begin{equation}\label{eq:balanced}
\left| |B\cap A| - \frac{b}{b+1}|A| \right| < \eps |A|
\end{equation}
holds for every $A\in\mathcal{H}$.
The following theorem provides a general winning criterion for Balancer.

\begin{thm}[Theorem 1.5 in~\cite{HKS2007}]\label{thm:balancer}
Let $(X,\mathcal{H})$ be a $k$-uniform hypergraph.
If $$b<\frac{1}{3}\sqrt[3]{\frac{k}{\log(|\mathcal{H}|k)}} ~~ \text{and} ~~ \eps>3\sqrt{\frac{\log(|\mathcal{H}|k)}{kb}}$$
while $k$ is sufficiently large,
then Balancer has a winning strategy for 
the $(b,1,\mathcal{H})$ $\eps$-Discrepancy game.
\end{thm}

With the above result in our hands, let us now turn to the proof of Lemma~\ref{lem:regularity_Maker}.

\begin{proof}[Proof of Lemma~\ref{lem:regularity_Maker}]

First note, that if $(A,B)_G$ has density $\alpha$,
then $(A,B)_\maker$ will have density $\frac{\alpha}{2}$
since Maker claims half of all edges.
Now, for every $S\subset A$, $T\subset B$ of size $\eps n$
we have
$
\alpha - \eps < d_G(S,T) < \alpha + \eps
$
and hence
$$  (\alpha - \eps)\eps^2n^2 < e_G(S,T) < (\alpha + \eps)\eps^2n^2 ~ .$$
For every such pair $(S,T)$ fix an arbitrary subset
$H(S,T)\subset e_G(S,T)$ of size 
$(\alpha - \eps)\eps^2n^2$, and let
$$
\mathcal{H} =
\Big\{
H(S,T):~ S\subset A,~ T\subset B ~ \text{of size }
\eps n
\Big\}~ .
$$
Maker plays as Balancer on the hypergraph 
$(X,\mathcal{H})$, where $X=\bigcup_{H(S,T)\in \mathcal{H}}H(S,T)$.
That is, whenever Breaker claims an edge belonging
to $X$, Maker (as Balancer) claims an edge according to the strategy for the 
$(1,1,\mathcal{H})$ $\eps$-Discrepancy game,
given by Theorem~\ref{thm:balancer}, for
$b=1$ and $k=(\alpha - \eps)\eps^2n^2$.
Whenever Breaker claims an edge not in $X$,
Maker does the same.

In order to see that Theorem~\ref{thm:balancer} can be applied, 
observe that $|\mathcal{H}| \leq \binom{n}{\eps n}^2 < 4^n$,
and hence
$$\frac{k}{\log(|\mathcal{H}|k)}
 = \Omega(n)
~~ \text{and} ~~
\frac{\log(|\mathcal{H}|k)}{kb} = O\left(\frac{1}{n}\right)~,
$$
which yields
$$b<\frac{1}{3}\sqrt[3]{\frac{k}{\log(|\mathcal{H}|k)}}
~~ \text{and} ~~
\eps>3\sqrt{\frac{\log(|\mathcal{H}|k)}{kb}}$$
provided $n$ is large enough.
Now, let $\maker$ denote Maker's graph at the end of the game. As a result of Maker's strategy, Maker ensures that
$$
\left| |E(\maker)\cap H(S,T)| - \frac{1}{2}|H(S,T)| \right| < \eps |H(S,T)| ~~ \text{for every }H(S,T)\in\mathcal{H},
$$
see inequality~(\ref{eq:balanced}).
From this, we obtain for every $S\subset A,~ T\subset B$ of size $\eps n$ that
\begin{align*}
e_\maker(S,T)
	& \geq |E(\maker)\cap H(S,T)|
	\geq \left(\frac{1}{2} - \eps \right)|H(S,T)|
\geq \left(\frac{1}{2} - \eps \right)(\alpha - \eps)\eps^2n^2\\
	& \Rightarrow
d_\maker(S,T) \geq \left(\frac{1}{2} - \eps \right)(\alpha - \eps)
	\geq \frac{1}{2}\alpha - 2\eps, ~ \text{and}\\
e_\maker(S,T)
	& \leq |E(\maker)\cap H(S,T)| + |E_G(S,T)\setminus H(S,T)|
	\leq \left(\frac{1}{2} + \eps \right)(\alpha - \eps)\eps^2n^2 + 2\eps^3 n^2\\
	& \Rightarrow
d_\maker(S,T) \leq \left(\frac{1}{2} + \eps \right)(\alpha - \eps)
	+ 2\eps \leq \frac{1}{2}\alpha + 3\eps, ~
\end{align*}
i.e.~$|d_\maker(S,T)-d_\maker(A,B)|<4\eps$. By a simple averaging argument the latter extends to all subsets $S,T$ of size at least $\eps n$. That is, $(A,B)_\maker$ is $4\eps$-regular with density at least $\frac\alpha2$. 
\end{proof}

Together with Lemma~\ref{lem:reg-turan} this is sufficient to prove Theorem~\ref{thm:hgame0}.
For Theorem~\ref{thm:H-game_Maker} Maker also needs to find many copies of small graphs within the random graph such that they can be combined using the regular pairs.

\subsection{Creating many $H$-copies on a random graph}
\label{subsec:many-H}

It is known (Theorem~16 in~\cite{NSS2016}) that for any graph $H$, which contains a cycle, there is a constant $C>0$ such that a.a.s.~the following holds when $p\geq Cn^{-1/m_2(H)}$: in the $(1:1)$ Maker-Breaker game on $G_{n,p}$,  
Maker has a strategy to occupy a copy of $H$. 

Another ingredient for Maker's strategy is to show that, if $p$ is slightly larger than mentioned above, Maker a.a.s.~has a strategy to occupy a copy of $H$ on every vertex set of almost linear size.

\begin{lemma}\label{lem:many-copies_Maker}
Let $H$ be any graph.
Then for every $\gamma>0$ there exists $\beta>0$ such that with $p \geq n^{-1/m_2(H) + \gamma}$ a random graph $G \sim G_{n,p}$ a.a.s.~satisfies the following property:
playing a $(1:1)$ Maker-Breaker game on $G$,
Maker has a strategy to occupy a subgraph of $G$
that has a copy of $H$ on every vertex set of size $n^{1-\beta}$.
\end{lemma}

The proof of the above lemma will follow mostly the argument given
in the papers of Bednarska and \L uczak~\cite{BL2000},
as well as Stojakovi\'c and Szab\'o~\cite{SS2005}.
We start with the following lemma. For this, note that
$G(n,M)$ denotes the random graph model, where we pick a
graph on $n$ vertices and with exactly $M$ edges
uniformly at random.

\begin{lemma}[Lemma 4 in~\cite{BL2000}]
Let $H$ be any graph containing a cycle,
then there exist constants $\beta>0$
and $c>0$ such that for every large enough $n$ the following holds:
if $M = 2n^{2-1/m_2(H)}$
then with probability at least $1-\exp(-c M)$ each subgraph of $G \sim G(n,M)$
with $(1-\beta)M$ edges contains a copy of $H$.
\end{lemma}

Note that in \cite[Lemma~4]{BL2000}
the probability of the good event happening is stated to be at least $\frac{2}{3}$; 
however with a closer look at the proof one actually sees 
that this probability is at least $1-\exp(-cM)$ for some positive constant $c$.
Further note that if we increase the number of edges of the random graph, 
but still delete at most the same number $2\beta n^{2-1/m_2(H)}$ of edges, 
then it does not become less likely to find copies of $H$. In particular, we obtain the following.

\begin{cor}\label{cor:BL}
Let $H$ be any graph containing a cycle,
then there exist constants $\beta>0$ and $c>0$
such that for every large enough $n$ the following holds:
if $M \geq 2n^{2-1/m_2(H)}$
then with probability at least $1-\exp(-c n^{2-1/m_2(H)})$ each subgraph of $G \sim G(n,M)$
with $M - \beta n^{2-1/m_2(H)}$ edges contains a copy of $H$.
\end{cor}

In fact, if the number of edges is increased slightly in the order of magnitude, then we can even find copies of $H$
on every vertex set of almost linear size.

\begin{cor}~\label{cor:BL2}
Let $H$ be any graph containing a cycle.
Then for every $\gamma>0$ there exists a constant 
$\beta>0$ such that the following holds a.a.s.~in $G \sim G(n,M)$ with
$M \geq n^{2-1/m_2(H) + \gamma/4}$:

for every vertex subset $A\subset [n]$ of size $n^{1-\beta}$
and every subgraph $F\subset G$ with $e(F)\leq n^{2-1/m_2(H) - \gamma/3}$
it holds that
$(G\setminus F)[A]$ contains a copy of $H$.
\end{cor}

\begin{proof}
Let constants $\beta$ and $c$ be chosen such that Corollary~\ref{cor:BL} can be applied and such that 
$0< \beta < \frac{\gamma}{12}$ and 
$(1-\beta)\left(2-\frac{1}{m_2(H)}\right)>1+\beta$.
To see that the latter is possible, note that $m_2(H)>1$ by the assumption on $H$.
From now on, whenever necessary, assume $n$ to be large enough.

Let $A\subset [n]$ be any vertex subset of size $n^{1-\beta}$,
and let $\E_A$ be the event
that, when generating $G \sim G(n,M)$, 
there exists a subgraph $F$ with 
$e(F)\leq n^{2-1/m_2(H) - \gamma/3}$
such that
$(G\setminus F)[A]$ does not contain a copy of $H$.
We will show in the following that
\begin{equation}\label{event_A}
\Prob(\E_A) \leq 2e^{-n^{1+\beta}}~ ;
\end{equation}
with a union bound over all choices for $A$ (the number of which is bounded by $2^n$)
the corollary then follows.
In order to prove inequality~(\ref{event_A}), let us observe first that
\begin{equation}\label{edges_A}
(1 - \beta)Mn^{-2\beta} \leq e_G(A) \leq (1 + \beta)Mn^{-2\beta}
\end{equation}
holds with probability at least $1-\exp(-n^{1+\beta})$. Indeed, the random variable
$e_G(A)$ has hypergeometric distribution 
with expectation
$$
\Exp(e_G(A))=M\cdot \frac{\binom{|A|}{2}}{\binom{n}{2}} 
= (1-o(1))
M\cdot \frac{|A|^2}{n^2}
= (1-o(1))Mn^{-2\beta}~ .
$$
Then, with Lemma~\ref{lem:hypergeometric}, we get
$$
\Prob\left( 
|e_G(A) - Mn^{-2\beta}| > \beta Mn^{-2\beta} \right)
\leq e^{-\beta^2Mn^{-2\beta}/4} < e^{-n^{1+\beta}}~ ,
$$
where for the last inequality we use that 
$2-\frac{1}{m_2(H)}>1$ and $\frac{\gamma}{4}-2\beta>\beta$.
Next, if we condition on (\ref{edges_A}), we obtain that
$G[A]$ is distributed according to $G(N,M^\ast)$ with
$N=n^{1-\beta}$ and 
\begin{align*}
M^\ast \geq (1 - \beta - o(1))Mn^{-2\beta} 
\geq N^{2-1/m_2(H) +\beta}~ ,
\end{align*}
where for the last inequality we use that 
$\frac{\gamma}{4}-2\beta>\beta$ and $n> N$.
It follows from Corollary~\ref{cor:BL} that with probability
at least 
$1-\exp(-c N^{2-1/m_2(H)}) \geq 1- \exp(-n^{1+\beta})$ each subgraph of $G[A]$
with $e_G(A) - \beta N^{2-1/m_2(H)}$ edges contains a copy of $H$.
Since  $n^{2-1/m_2(H) - \gamma/3} < \beta N^{2-1/m_2(H)}$ holds
for large $n$ by having $\beta < \frac{\gamma}{12}$, this in particular means 
that $(G\setminus F)[A]$ contains a copy of $H$ for every 
graph $F$ satisfying $e(F)\leq n^{2-1/m_2(H) - \gamma/3}$.
Hence, inequality~(\ref{event_A}) is proven.
\end{proof}

If $H$ does not contain a cycle, we can get the same conclusion.

\begin{cor}~\label{cor:BL3}
	Let $H$ be any graph with $m_2(H)=1$.
	Then for every $\gamma>0$ there exists a constant 
	$\beta>0$ such that the following holds a.a.s.~in $G \sim G(n,M)$ with
	$M \geq n^{1 + \gamma/4}$:

	for every vertex subset $A\subset [n]$ of size $n^{1-\beta}$
	and every subgraph $F\subset G$ with $e(F)\leq n^{1 - \gamma/3}$
	it holds that
	$(G\setminus F)[A]$ contains a copy of $H$.
\end{cor}

\begin{proof}
	Let $H$ be any graph with $m_2(H)=1$ and let $\gamma>0$.
	We construct a graph $H'$ by adding a disjoint cycle of length $\left\lceil \frac{8}{\gamma}+2 \right\rceil$ to $H$ and note that $m_2(H') \le 1+\frac{\gamma}{8}$ and, therefore, $2-\frac{1}{m_2(H')} \le 1+  \frac{\gamma}{8}$.
	Now we let $\beta>0$ be given by Corollary~\ref{cor:BL2} with input $H'$ and $\frac{\gamma}{2}$.
	Then a.a.s.~in $G(n,M)$ with $M \ge n^{1+\gamma/4} \ge n^{2-1/m_2(H')+\gamma/8}$ for every vertex subset $A \subseteq [n]$ of size $n^{1-\beta}$ and every subgraph $F \subseteq G$ with $e(F) \le n^{2-1/m_2(H')-\gamma/6}$ it holds that $(G\setminus F)[A]$ contains a copy of $H'$ and, thus, also of $H$.
	Since $n^{1-\gamma/3} \le  n^{2-1/m_2(H')-\gamma/6}$ the statement follows.
\end{proof}

Using Corollaries~\ref{cor:BL2} and~\ref{cor:BL3} we finally can prove
Lemma~\ref{lem:many-copies_Maker}.

\begin{proof}[Proof of Lemma~\ref{lem:many-copies_Maker}]
The proof idea for the lemma is 
similar to the one
given 
by Bednarska and \L uczak~\cite{BL2000},
or Stojakovi\'c and Szab\'o~\cite{SS2005}
for the discussion of a Maker's 
strategy in the $H$-game.
We will prove that a.a.s.~Maker has a strategy for
occupying a graph as desired
when playing on a random graph $G' \sim G(n,M')$
where $M'=p\binom{n}{2}$. The result then follows for $G_{n,p}$ as the property we are looking for is monotone increasing (see e.g.~Proposition~1.12 in~\cite{JLR2000}).

Maker's strategy is to play randomly. That is, in each of her moves, Maker takes an edge from $G'$ uniformly at random from all the edges she has not taken in previous rounds. If this edge is free, then she claims it; otherwise she declares her move as a failure and simply skips her move. We will show that against any fixed Breaker's strategy this random strategy a.a.s. leads to a subgraph $H$. From this it then follows that a.a.s.~there must also exist a deterministic strategy for winning the $H$-game (see e.g.~\cite{BL2000,HKSS_PosGames}).

In order to prove that Maker succeeds a.a.s.~let us consider only the first $M=n^{2-1/m_2(H)+\gamma/4}$ rounds of the game. At the end of the $M^{\text{th}}$ round we have that
all edges taken by Maker (but not necessarily being claimed by her due to a failure) 
form a graph distributed from $G(n,M)$. 
Hence,
using Corollaries~\ref{cor:BL2} and~\ref{cor:BL3}, it is enough to prove that a.a.s.~we have at most $n^{2-1/m_2(H)-\gamma/3}=n^{-7\gamma/12}M$ failures by the end of round $M$.

In order to see that the number of failures can be bounded this way, notice that the board size is $e(G')=M'\geq \frac{1}{3} n^{2-1/m_2(H)+\gamma}$ and hence, until round $M$, each round happens to be a failure with probability at most
$$
\frac{M}{M'-M} \leq 4n^{-3\gamma/4}~ .
$$
Therefore, the expected number of failures up to round $M$ can be bounded from above by
$4n^{-3\gamma/4}M$. Applying Markov's inequality~(Lemma~\ref{lem:markov}) we thus obtain that with probability at most $4n^{-\gamma/6}$ there happen to be more than
$n^{-7\gamma/12}M$ failures.
This proves the lemma.
\end{proof}

\subsection{A general result for $H$-games}
\label{subsec:H-game-strategy}
The goal of this subsection is to prove Theorem~\ref{thm:H-game_Maker}.

\begin{proof}[Proof of Theorem~\ref{thm:H-game_Maker}]
	Let $\gamma>0$ and $r \ge 2$ be any integer, let 
	$\alpha \in ( \frac{r-2}{r-1},\frac{r-1}{r} ]$, 
	$H$ be a fixed graph with $m_2^{(r)}(H)>0$, and let $G_\alpha$ be any 
	$n$-vertex graph with minimum degree at least 
	$\alpha n$ and \[p \ge n^{-1/m_2^{(r)}(H)+\gamma}\, .\]
	By the definition of $m_2^{(r)}(H)$, we find a 
	partition $P_1,\dots,P_r$ of $V(H)$
	such that $m_2(H[P_i])\leq m_2^{(r)}(H)$
	for every $1 \le i \le r$.	
	For short, let $H_i=H[P_i]$ for $1 \le i \le r$,
	$H^{(i)} = H[P_1 \cup \dots \cup P_i]$ for 
	$0 \le i \le r$, and set $\ell= \max_i \{ |P_i| \}$.
	Then $H^{(0)}$ is the empty graph and $H^{(r)}=H$, and, moreover, for $1 \le i \le r-1$, $H^{(i+1)}$ is contained in the graph that we get when we take the vertex disjoint union of $H^{(i)}$ and $H_{i+1}$ and add all edges between both graphs.
	
	We apply Lemma~\ref{lem:many-copies_Maker} with $\frac{\gamma}{2}$ for each $H_i$ such that $v(H_i) \ge 2$, resulting in some output $\beta_i=\beta(H_i)$, and set $\beta_i=1$ if $v(H_i)=1$.
	Then we set $\beta = \min_{i\in [r]} \beta_i$.
	Further, we choose $\delta >0$ such that $\alpha \ge \frac{r-2}{r-1}+\delta$.
	From Lemma~\ref{lem:reg_drc} with input
	$r$, $\frac{\delta}{2^{r+4}}$, 
	$\frac{\beta}{2}$, 
	and $\ell$ we obtain a constant $\nu \le \frac{1}{2}$.
	We then choose a positive constant 
	$\eps\le \min \{  \frac{\delta}{2^{r+10}} , \frac{\nu^r}{4}\}$, obtain 
	$m_0$ from Lemma~\ref{lem:reg_drc} with input $\eps$, 
	and obtain $\eta$ and $n_0$ from Lemma~\ref{lem:reg-turan} with inputs $\eps $ and $\delta$.
	We let $n$ be large enough for Lemma~\ref{lem:regularity_Maker} and for the application of the other lemmas, to ensure that $n \ge n_0$, $\nu^r \eta n \ge m_0$, and $(\nu^r \eta n)^{1-\beta/2} \ge n^{1-\beta}$.
	
\medskip

	Before revealing $G \sim G_{n,p}$, we apply
	Lemma~\ref{lem:reg-turan} to the graph 
	$G_{\alpha}$, and we find pairwise disjoint sets 
	$V_1,\dots,V_r$ of size at least $\eta n$ such that 
	$(V_i,V_j)$ is $\eps$-regular with density equal to some constant 
	$\delta_{i,j}\geq \frac{\delta}{2}$, 
	for every $1 \le i < j \le r$.
	We denote by $G_1$ the $r$-partite subgraph 
	of $G_{\alpha}$ with classes $V_1,\dots,V_r$.
	Finally, revealing the edges of $G \sim G_{n,p}$,
	we let $G_2$ denote the union of all graphs 
	$G[V_i]$ with $1 \le i \le r$.
	Then	 a.a.s.~we have that each $G[V_i]$ satisfies 
	the conclusion of Lemma~\ref{lem:many-copies_Maker} 
	applied to the graph $H_i$, since
	$p \ge |V_i|^{-1/m_2(H_i) + \gamma/2}$
	for large enough $n$. 
	From now on we will condition on these conclusions.
	Next we will describe a strategy for Maker in an unbiased game on $G_{\alpha}\cup G_{n,p}$, 
	and show that she can follow that strategy, leading to a copy of $H$ in her final 
	graph.
	
	\medskip

{\bf Strategy description:}	
	Maker plays as follows.
	Consider the edge-disjoint boards
	$E_{G_2}(V_i)$ for all $i\in [r]$, and 
	$E_{G_1}(V_i,V_j)$ for all $1\leq i<j\leq r$.
	Maker always plays on the same board as Breaker.
	On each of the boards $E_{G_1}(V_i,V_j)$ Maker 
	follows the strategy guaranteed by 
	Lemma~\ref{lem:regularity_Maker} and thus 
	occupies a subgraph $\maker_1 \subseteq G_1$ such that 
	$(V_i,V_j)_{\maker_1}$ 
	is $4\eps$-regular with density $\frac{\delta_{i,j}}{4}$
	for every $1 \le i<j \le r$.
	On each of the boards $E_{G_2}(V_i)$ 
	Maker follows the strategy guaranteed by 
	Lemma~\ref{lem:many-copies_Maker}
	and thus obtains a 
	subgraph $\maker_2 \subseteq G_2$ such that for 
	every $1 \le i \le r$ and any $U \subset V_i$ of 
	size at least $n^{1-\beta}$ there is 
	a copy of $H_i$ in $\maker_2[U]$.

	\medskip

{\bf Strategy discussion:} Maker can follow her 
	strategy	 by the conclusions of 
	Lemma~\ref{lem:regularity_Maker} and
	Lemma~\ref{lem:many-copies_Maker}.
	Hence, it remains to show that, by following the 
	strategy, Maker obtains a copy of $H$.	
	We will build this copy inductively in the order $H^{(0)},H^{(1)},\dots,H^{(r)}$.
	For $0 \le s \le r$ we want to find a copy of $H^{(s)}$ in $(\maker_1 \cup \maker_2)[V_1 \cup \dots \cup V_s]$ such that in the common neighbourhood (with respect to $\maker_1$) of the vertices of this copy there are pairwise disjoint sets $V^{(s)}_{s+1}\subset V_{s+1},\dots,V^{(s)}_{r}\subset V_r$ of size $\nu^s \eta n$ such that $(V^{(s)}_i,V^{(s)}_j)_{\maker_1}$ is $4 \eps \nu^{-s}$-regular with density at least $\frac{\delta}{2^{s+3}}$ for $s+1 \le i < j \le r$.
	Observe, that the above already holds for $s=0$ when we choose $V^{(0)}_i=V_i$ for $1 \le i \le r$. 
	We are finished when we arrive at $s=r$, since we then have a copy of $H^{(r)}=H$ in $\maker_1 \cup \maker_2$.
	
	Assume the above holds for some $0 \le s \le r-1$.
	We already have a copy of $H^{(s)}$ with the described properties.
	Next we apply Lemma~\ref{lem:reg_drc} to obtain a set $U \subseteq V^{(s)}_{s+1}$ of size $(\nu^s \eta n)^{1-\beta/2} \ge n^{1-\beta}$ such that any $\ell$ vertices from $U$ have at least $\nu^{s+1} \eta n$ common neighbours in each of the sets $V^{(s)}_i$ for $s+2 \le i \le r$, which we denote by $V^{(s+1)}_i$.
	Then by Lemma~\ref{lem:slicing} and the choice of $\eps$ we obtain that all pairs $(V^{(s+1)}_i,V^{(s+1)}_j)_{\maker_1}$ are $4 \eps \nu^{-s-1}$-regular with density at least $\frac{\delta}{2^{s+4}}$ for $s+2 \le i < j \le r$.
	Moreover, by the game on $E_{G_2}(V_{s+1})$ there is a copy of $H_{s+1}$ in $\maker_2[U]$.
	By construction together with the copy of $H^{(s)}$ this gives a copy of $H^{(s+1)}$ with the desired properties.
\end{proof}

\begin{rem}\label{rem:r-chromatic}
	The proof of Theorem~\ref{thm:hgame0} follows analogously.
	Indeed, when $H$ has chromatic number $r$, we have $m_2^{(r)}(H)=0$ and all $H_i$ consist of isolated vertices.
	Then Maker only needs to play on $G_1$ to obtain a copy of $H$, exactly as described above.
\end{rem}

\begin{rem}
	We do not need the $\gamma$ in Theorem~\ref{thm:H-game_Maker} when there is only a single graph $H'$ in $H_1,\dots,H_k$ that satisfies $m_2(H')=m_2^{(r)}(H)$.
	This is because in our construction we can find this copy in the last set $V_r$ using a version of Lemma~\ref{lem:many-copies_Maker} that only requires $p \ge C n^{-1/m_2(H)}$ and guarantees a copy of $H'$ in every set of size $\beta n$ for some not too small constant $\beta>0$.
	
	We also do not need the $\gamma$ when $G_\alpha$ is the $r$-partite Tur\'an-graph.
	Intuitively the behaviour should be the same as in the general case, and therefore we believe that $\gamma$ is not needed in any case.
\end{rem}

Finally, the Waiter-Client result for the $H$-game can be proven fairly easily.

\begin{proof}[Proof of Theorem~\ref{thm:H-game_WC}]

For any graph $G$ set $k(G)=2^{e(G)}$.
Let $F$ be the vertex disjoint union of $k(H)$ copies
of the graph $H$. It holds that $m^{(r)}(F)=m^{(r)}(H)$.
Hence, following Theorem 2.1 from~\cite{KST_SmoothedAnalysis},
we know that a.a.s. a graph $G\sim G_\alpha \cup G_{n,p}$
contains a copy of $F$. 

It thus remains to show that playing on $F$,
Waiter has a strategy to force a copy of $H$ in Client's graph.
This can be done by induction on $e(H)$.
Let $e$ be any edge in $E(H)$, and denote with $e_1,\ldots,e_{k(H)}$ the copies of $e$ in the copies of $H$ in $F$.
Then Waiter can offer these edges in pairs,
until Client claimed $\frac{k(H)}{2}=k(H-e)$ such edges.
Immediately afterwards, there are $k(H-e)$ copies of $H$ in which Client already claimed the copy of $e$ and in which Waiter does not occupy any edge yet. Thus the problem is reduced to force a copy of $H-e$ on the disjoint union of $k(H-e)$ copies of $H-e$, which can be shown by induction.
\end{proof}

\bigskip


\section{Concluding remarks}
\label{sec:conc}

\subsection{Optimality of Theorem~\ref{thm:H-game_Maker}}
In this paper we proved optimal results for the $k$-vertex-connectivity and Hamiltonicity Maker-Breaker games in randomly perturbed graphs.
It remains to discuss when Theorem~\ref{thm:H-game_Maker} is optimal up to the $\gamma$.
More precisely, the question is when $p \le cn^{-1/m_2^{(r)}(H)}$ is enough to ensure that a.a.s.~Breaker has a winning strategy in the $H$-game on $G_\alpha \cup \gnp$
for some choice of $G_{\alpha}$.
For this, consider $G_{\alpha}$ to be the $r$-partite Tur\'an graph on $n$ vertices, with vertex classes $V_1,\ldots,V_r$. If Maker wants to create a copy
of $H$ in the game on $G_{\alpha}\cup G_{n,p}$,
then she must have a strategy
for creating some subgraph $H'\subseteq H$
with $m_2(H') \geq m_2^{(r)}(H)$ on one of the sets $V_i$.
Therefore, it suffices for Breaker to ensure that,
when playing on $G_{n,p}$, Maker loses the Maker-Breaker $\mathcal{H}$-game, where the family $\mathcal{H}$ of winning sets consists of all copies of all such subgraphs $H'$.
A.a.s.~Breaker has a winning strategy if $\mathcal{H}$ does not contain any exceptional graph $H'$, for which the Maker-Breaker $H'$-game threshold is not known, or is not of order $n^{-1/m_2(H')}$, or $H'=K_4$.
This can be proven analogously to the Maker-Breaker $H$-game on random graphs~\cite[Theorem~2]{NSS2016}.
The reason that $K_4$ is excluded is because this case is treated separately in~\cite[Lemma~2.1]{MS2014}  and there is no immediate way to combine the proofs.
When the threshold of $H'$ is not of the order $n^{-1/m_2(H')}$, e.g.~when $H$ is a tree or a triangle, the bound from Theorem~\ref{thm:H-game_Maker} can be significantly improved as we have demonstrated in Proposition~\ref{prop:oddcycle} and~\ref{prop:book}.
It would be interesting to see if, more generally, $p=\omega(n^{-2})$ is always sufficient when there exists an edge $e \in E(H)$ such that $m^{(r)}(H-e)=0$.
This also generalises to many other graphs $H$ with $m^{(r)}_2(H) = 1$ and we give more details for one example in the next section.

\subsection{$H$-game for small graphs}

$H=K_4$ is the smallest graph for which we do not know the threshold probability for winning the $H$-game on 
$G_\alpha\cup G_{n,p}$ with $\alpha \in \left(0,\frac{1}{2} \right)$.
Note that it follows from~\cite{KST_SmoothedAnalysis} that a.a.s.~in $G_\alpha \cup \gnp$ there is a copy of $K_4$, provided that $p=\omega(n^{-2})$.

However, even when $p = o\left(n^{-5/4}\right)$,
there is a.a.s.~an easy strategy for Breaker to ensure that Maker does not get a copy of $K_4$.
Indeed, in this case, let $G_\alpha$ be any bipartite graph with $\delta(G_{\alpha})\geq \alpha n$ and with partition classes $A$ and $B$, and note that in $\gnp$ a.a.s.~all components are trees on at most $4$ vertices.
Conditioning on this event and assuming that all edges in $G[A]$ and $G[B]$ already belong to Maker, it is a simple case distinction to check that Breaker has a strategy on the edges of $G[A,B]$ to prevent Maker from claiming a copy of $K_4$.
It is plausible that with a more involved strategy for Breaker, also responding on the edges inside $A$ and $B$, the bound on $p$ can be increased further.

On the other hand, there also is a simple strategy for Maker when $p = \omega\left(n^{-8/7}\right)$.
Here, we can use Theorem~2.1 from~\cite{KST_SmoothedAnalysis} to show that $G_\alpha \cup \gnp$ will a.a.s. contain a graph created from a complete bipartite graph by adding to both partition classes many vertex-disjoint copies of stars with $7$ edges. Maker can easily claim $4$ edges from such a star, when she is the first player to claim an edge. By claiming multiple stars with $4$ edges in each partition class, she can ensure that she claims a pair of such stars such that all edges between the two stars are still free. Then she can restrict the game to the complete bipartite graph between those two stars, and again it is an easy case distinction to show that Maker can complete a copy of $K_4$ within the next $5$ moves.

It is unclear if this strategy can be significantly improved or how far the Breaker argument can be pushed.

\subsection{Maker-Breaker $K_t$-factor game}
In an $n$-vertex graph with $t | n$ a $K_t$-factor is the disjoint union of $\frac{n}{k}$ copies of $K_t$.
In this section we implicitly assume $k|n$.
Recently, 
Allen, B\"ottcher, Kohayakawa, Naves, and Person~\cite{ABKNP2017} determined the threshold bias for the 
Maker-Breaker $K_t$-factor game on $K_n$, for $t\in \{3,4\}$, up to a logarithmic factor.
Even more recently, Liebenau and Nenadov~\cite{liebenau2020threshold} determined the threshold
for every $t\geq 3$. They proved that for $t \ge 3$ there are constants $c,C>0$ such that Breaker wins if $b<cn^{2/(t+2)}$ and Maker wins if $b>Cn^{2/(t+2)}$.

We briefly summarise what is known about the appearance of a $K_t$-factor in the other models that we discussed.
Hajnal and Szemer\'{e}di~\cite{hajnal1970proof} proved that any $n$-vertex graph with minimum degree at least $\left(1-\frac{1}{t}\right)n$ contains a $K_t$-factor.
Johannson, Kahn, and Vu~\cite{JKV} showed that $n^{-2/t} \log^{2/(t^2-t)}n$ gives the threshold for the containment of a $K_t$-factor in $\gnp$.
In the perturbed model $G_\alpha \cup \gnp$ it was shown by Balogh, Treglown, and Wagner~\cite{balogh2019tilings} that for any $\alpha>0$ a probability of $p = \omega (n^{-2/t})$ is a.a.s.~sufficient, while for large $\alpha$ more precise results are known~\cite{han2019tilings}.

The $\log$-term in~\cite{JKV} is needed to ensure that every vertex is contained in a copy of $K_t$, which, of course, is immediate in the perturbed model.
Similarly, when we consider the Maker-Breaker $K_t$-factor game on $\gnp$, Maker needs to ensure that in her graph every vertex is contained in a copy of $K_t$.
Since each vertex has roughly $np$ neighbours and there has to be a copy of $K_{t-1}$ in each neighbourhood, we thus require that Maker wins the $K_{t-1}$-game in $G_{np,p}$. If $k\geq 5$ this implies that we need $p \ge C(np)^{-2/t}$ and, thus, $p \ge C' n^{-2/(t+2)}$.
Note that this probabilistic intuition aligns with the threshold bias in~\cite{liebenau2020threshold} discussed above.

While it is not known if the threshold probability for the $K_t$-factor game in $\gnp$ is $n^{-2/(t+2)}$, it would be interesting to investigate this in the perturbed model.
More precisely, we ask which lower bound on $p$ is sufficient such that Maker a.a.s.~has a winning strategy in the Maker-Breaker $K_t$-factor game on $G_\alpha \cup \gnp$.
When $p = \omega(n^{-2/(t+1)})$ Maker a.a.s.~has a strategy such that every set of linear size in $\gnp$ contains a copy of $K_t$ (similar as in Corollary~\ref{cor:BL2}).
Together with the deterministic graph this could be sufficient to ensure that Maker is able to create a $K_t$-factor.
This is particularly interesting, because this probability differs significantly from the one needed for
the $K_t$-factor game on $G_{n,p}$.

\bigskip




\begin{thebibliography}{alpha}

\bibitem{ABKNP2017}
P.~Allen, J.~B\"ottcher, Y.~Kohayakawa, H.~Naves, and Y.~Person, 
{\em Making spanning graphs}, 
arXiv preprint, arXiv:1711.05311 (2017).

\bibitem{AHK2010}
N.~Alon, D.~Hefetz, and M.~Krivelevich, 
{\em Playing to retain the advantage}, 
Combinatorics, Probability and Computing 19(4) (2010), 481--491.

\bibitem{balogh2019tilings}
J.~Balogh, A.~Treglown, and A.~Z.~Wagner,
{\em Tilings in randomly perturbed dense graphs},
Combinatorics, Probability and Computing 28(2) (2019), 159--176.

\bibitem{Beck1982}
J.~Beck,
{\em Remarks on positional games. I}, 
Acta Mathematica Academiae Scientiarum Hungarica 40 (1-2), 65--71.

\bibitem{BeckBook}
J.~Beck, 
{\bf Combinatorial games: Tic-Tac-Toe theory}, 
Encyclopedia of Mathematics and its Applications 114, Cambridge University Press (2008).

\bibitem{B2013}
M.~Bednarska-Bzd\c{e}ga, 
{\em On weight function methods in Chooser-Picker games}, 
Theoretical Computer Science 475 (2013), 21--33.

\bibitem{BHKL2016}
M.~Bednarska-Bzd\c{e}ga, D.~Hefetz, M.~Krivelevich, and T.~\L uczak, 
{\em Manipulative waiters with probabilistic intuition}, 
Combinatorics, Probability and Computing 25(6) (2016), 823--849.

\bibitem{BL2000}
M.~Bednarska and T.~\L uczak,
{\em Biased positional games for which random strategies are nearly optimal}, 
Combinatorica 20(4) (2000), 477--488.

\bibitem{BFHK_MBhitting}
S.~Ben-Shimon, A.~Ferber, D.~Hefetz, and M.~Krivelevich,
{\em Hitting time results for Maker-Breaker games},
Random Structures \& Algorithms 41(1) (2012), 23--46.

\bibitem{BFKM2004}
T.~Bohman, A.~Frieze, M.~Krivelevich, and R.~Martin, 
{\em Adding random edges to dense graphs}, 
Random Structures \& Algorithms 24 (2004), 105--117. 

\bibitem{BFM_HAM}
T.~Bohman, A.~Frieze, and R.~Martin, 
{\em How many random edges make a dense graph Hamiltonian?},
Random Structures \& Algorithms 22(1) (2003), 33--42.

\bibitem{Bollobas}
B.~Bollob\'as, 
{\bf Random graphs}, 
Vol. 73, Cambridge University Press, 2001.

\bibitem{B_HAM}
B.~Bollob\'as,
{\em The evolution of sparse graphs},
Graph Theory and Combinatorics (Cambridge 1983), Academic Press, London (1984) 35--57.

\bibitem{CE1978}
V.~Chv\'{a}tal and P.~Erd\H{o}s,
{\em Biased positional games}, 
Annals of Discrete Mathematics (2) (1978), 221--229.


\bibitem{DT_PeturbedRamsey}
S.~Das and A.~Treglown, 
{\em Ramsey properties of randomly perturbed graphs: cliques and cycles}, 
Combinatorics, Probability and Computing, 1--38 (2020).

\bibitem{DK2016}
O.~Dean and M.~Krivelevich,
{\em Client-Waiter games on complete and random graphs}, 
The Electronic Journal of Combinatorics 23(4) (2016), P4.38.

\bibitem{D_HAM}
G.~A.~Dirac, 
{\em Some theorems on abstract graphs},
Proceedings of the London Mathematical Society 3(1) (1952), 69--81.

\bibitem{hajnal1970proof}
A.~Hajnal and E.~Szemer{\'e}di,
{\em Proof of a conjecture of P.~Erd\H{o}s},
Combinatorial theory and its applications 2 (1970), 601--623.

\bibitem{ER_connected}
P.~Erd\H{o}s and A.~R\'enyi,
{\em On the strength of connectedness of a random graph},
Acta Mathematica Academiae Scientiarum Hungarica 12(1-2) (1964), 261--267.

\bibitem{ES_extremal}
P.~Erd\H{o}s and A.~H.~Stone,
{\em On the structure of linear graphs},
Bulletin of the American Mathematical Society 52 (1946), 1087--1091.

\bibitem{FGKN_BiasedGames}
A.~Ferber, R.~Glebov, M.~Krivelevich, and A.~Naor, 
{\em Biased games on random boards}, 
Random Structures \& Algorithms 46(4) (2015), 651--676.

\bibitem{FKN2015}
A.~Ferber, M.~Krivelevich, and H.~Naves,
{\em Generating random graphs in biased Maker-Breaker games}, 
Random Structures \& Algorithms 47 (2015), 615--634.

\bibitem{FS_DRC}
J.~Fox and B.~Sudakov, 
{\em Dependent random choice}, 
Random Structures \& Algorithms 38(1-2) (2011), 68--99.

\bibitem{GS_MinDegree}
H.~Gebauer and T.~Szab\'{o}, 
{\em Asymptotic random graph intuition for the biased connectivity game}, 
Random Structures \& Algorithms 35(4) (2009), 431--443.

\bibitem{han2019tilings}
J.~Han, P.~Morris, and A.~Treglown,
{\em Tilings in randomly perturbed graphs: bridging the gap between Hajnal-Szemer\'{e}di and Johansson-Kahn-Vu},
arXiv preprint, arXiv:1904.09930 (2019).

\bibitem{HKSS_PosGames}
D.~Hefetz, M.~Krivelevich, M.~Stojakovi\'c, and T.~Szab\'o, 
{\bf Positional games}, 
Basel: Birkh\"auser (2014).

\bibitem{HKSS2009}
D.~Hefetz, M.~Krivelevich, M.~Stojakovi\'c, and T.~Szab\'o, 
{\em A sharp threshold for the Hamilton cycle Maker-Breaker game}, 
Random Structures \& Algorithms 34(1) (2009), 112--122.

\bibitem{HKS2007}
D.~Hefetz, M.~Krivelevich, and T.~Szab\'o, 
{\em Bart-Moe games, JumbleG and discrepancy},
European Journal of Combinatorics, 28(4) (2007), 1131--1143.

\bibitem{HKS2009}
D.~Hefetz, M.~Krivelevich, and T.~Szab\'o,
{\em Hamilton cycles in highly connected and expanding graphs}, 
Combinatorica 29 (2009), 547--568.

\bibitem{HKT2017}
D.~Hefetz, M.~Krivelevich, and W.~E.~Tan, 
{\em Waiter-Client and Client-Waiter Hamiltonicity games on random graphs}, 
European Journal of Combinatorics 63 (2017), 
26--43.

\bibitem{Hoeffding1963}
W.~Hoeffding, 
{\em Probability Inequalities for Sums of Bounded Random Variables}, 
Journal of the American Statistical Association (1963), 13--30.

\bibitem{JLR2000}
S.~Janson, T.~\L uczak, and A.~Ruci\'nski, 
{\bf Random graphs},
Wiley, New York, 2000.

\bibitem{JKV}
A.~Johansson, J.~Kahn, and V.~Vu,
{\em Factors in random graphs},
Random Structures \& Algorithms, 33(1) (2008), 1--28.

\bibitem{Karger1993}
D.~Karger, 
{\em Global min-cuts and other ramifications of a simple min-cut algorithm}, 
Proceedings of the 4th Annual ACM-SIAM Symposium on Discrete Algorithms (1993), 21--30.

\bibitem{KS_RegularityApplications}
J.~Koml\'{o}s and M.~Simonovits, 
{\em Szemer\'{e}di's regularity lemma and its applications in graph theory}, 
Combinatorics, Paul Erd\H{o}s is Eighty (Volume 2), Keszthely (Hungary), 1993, (D. Mikl\'{o}s, V. T. S\'{o}s, T. Sz\H{o}nyi eds.), Bolyai Math. Stud., Budapest (1996), 295--352.

\bibitem{K_HAM}
A.~D.~Korshunov,
{\em Solution of a problem of Erd\H{o}s and Renyi on Hamiltonian cycles in nonoriented graphs},
Doklady Akademii Nauk 228(3) (1976), 
529--532.

\bibitem{Kriv_Ham}
M.~Krivelevich, 
{\em The critical bias for the Hamiltonicity game is (1+o(1))n/ln n}, 
Journal of the American Mathematical Society, 24(1) (2011), 125--131.

\bibitem{KLS_dirac}
M.~Krivelevich, C.~Lee, and B.~Sudakov,
{\em Robust Hamiltonicity of Dirac graphs},
Transactions of the American Mathematical Society 366(6) (2014), 3095--3130.

\bibitem{KST_SmoothedAnalysis}
M.~Krivelevich, B.~Sudakov, and P.~Tetali, 
{\em On smoothed analysis in dense graphs and formulas}, 
Random Structures \& Algorithms 29(2) (2006), 180--193.


\bibitem{liebenau2020threshold}
A.~Liebenau and R.~Nenadov,
{\em The threshold bias of the clique-factor game},
arXiv preprint, arXiv:2002.02578 (2020).


\bibitem{MS2014}
T.~M\"uller and M.~Stojakovi\'c,
{\em A threshold for the Maker-Breaker clique game},
Random Structures \& Algorithms 45(2) (2014), 318--341.

\bibitem{NSS2016}
R.~Nenadov, A.~Steger, and M.~Stojakovi\'c, 
{\em On the threshold for the Maker-Breaker H-game}, 
Random Structures \& Algorithms 49(3) (2016), 558--578.

\bibitem{P_HAM}
L.~P\'osa, 
{\em Hamiltonian circuits in random graphs},
Discrete Mathematics 14(4) (1976), 359--364.

\bibitem{SS2005}
M.~Stojakovi\'c and T.~Szab\'o,
{\em Positional games on random graphs}, 
Random Structures \& Algorithms 26(1-2) (2005), 204--223.

\bibitem{Tan2017}
W.~E.~Tan, 
{\em Waiter-Client and Client-Waiter games}, 
Doctoral dissertation, University of Birmingham (2017).

\bibitem{Turan}
P.~Tur\'an, 
{\em Eine Extremalaufgabe aus der Graphentheorie}, 
Mat. Fiz. Lapok (48) (1941), 436--452.

\bibitem{W2001}
D.~B.~West, 
{\bf Introduction to Graph Theory}, 
Prentice Hall (2001).

\end{thebibliography}
\end{document}